\documentclass[10pt]{amsart}

\usepackage{graphicx}

\usepackage{color}
\usepackage[USenglish]{babel}
\usepackage{latexsym}
\usepackage{amsmath}
\usepackage{amsfonts}
\usepackage{amssymb}
\usepackage{esint}
\usepackage[all]{xy}
\renewcommand{\a }{\alpha }
\renewcommand{\b }{\beta }
\renewcommand{\d}{\delta }
\newcommand{\D }{\Delta }

\renewcommand{\l }{\lambda }

\newcommand{\n }{\nabla }

\newcommand{\s }{\sigma }

\renewcommand{\o }{\omega }
\renewcommand{\O }{\Omega }

\newcommand{\ov}{\overline}

\def\o{\omega}
\def\p{\partial}

\newcommand{\wtilde }{\widetilde}

\newcommand{\be}{\begin{equation}}
\newcommand{\ee}{\end{equation}}

\newcommand{\R}{\mathbb{R}}
\renewcommand{\S}{\mathbb{S}}

\newcommand{\dis}{\displaystyle}

\newtheorem{theorem}{Theorem}[section]
\newtheorem{proposition}[theorem]{Proposition}
\newtheorem{definition}{Definition}[section]

\newtheorem{example}[theorem]{Example}

\newcommand{\bpr}{\begin{proposition}}
\newcommand{\epr}{\end{proposition}}
\newcommand{\bex}{\begin{example}\rm}
\newcommand{\eex}{\end{example}}

\begin{document}

\newtheorem{lem}{Lemma}[section]
\newtheorem{pro}[lem]{Proposition}
\newtheorem{thm}[lem]{Theorem}
\newtheorem{rem}[lem]{Remark}
\newtheorem{cor}[lem]{Corollary}
\newtheorem{df}[lem]{Definition}

\title[Singular Sphere Covering Inequality and Liouville-type problems]
{A singular Sphere Covering Inequality: \\ uniqueness and symmetry of solutions to \\ singular Liouville-type equations.}

\author{Daniele Bartolucci, Changfeng Gui, Aleks Jevnikar, Amir Moradifam}

\address{Daniele Bartolucci, Department of Mathematics, University of Rome {\it "Tor Vergata"},  Via della Ricerca Scientifica 1, 00133 Roma, Italy.}
\email{bartoluc@mat.uniroma2.it}

\address{Changfeng Gui,~Department of Mathematics, University of Texas at San Antonio, Texas, USA.}
\email{changfeng.gui@utsa.edu}

\address{Aleks Jevnikar,~University of Rome `Tor Vergata', Via della Ricerca Scientifica 1, 00133 Roma, Italy.}
\email{jevnikar@mat.uniroma2.it}

\address{Amir Moradifam,~Department of Mathematics, University of California, Riverside, California, USA.}
\email{moradifam@math.ucr.edu}

\thanks{D.B. and A.J. are partially supported by FIRB project "{\em Analysis and Beyond}",  by PRIN project 2012, ERC PE1\_11, "{\em Variational and perturbative aspects in nonlinear differential problems}", and by the Consolidate the Foundations project 2015 (sponsored by Univ. of Rome "Tor Vergata"),
"{\em Nonlinear Differential Problems and their Applications}". C.G. is partially supported by NSF grant DMS-1601885 and NSFC grant No 11371128, and A.M. is supported in part by NSF grant DMS-1715850.}



\begin{abstract}
We derive a singular version of the Sphere Covering Inequality which was recently introduced in \cite{GM1}, 
suitable for treating singular Liouville-type problems with superharmonic weights. As an application we deduce new uniqueness 
results for solutions of the singular mean field equation both on spheres and on bounded domains, as well as new self-contained proofs of 
previously known results, such as the uniqueness of spherical convex polytopes first established in \cite{lt}. Furthermore, 
we derive new symmetry results for the spherical Onsager vortex equation. 
\end{abstract}

\maketitle

\section{Introduction}

\medskip

We are concerned with a class of elliptic equations with exponential nonlinearities, namely the following Liouville-type equation,
\begin{equation} \label{eq:liouv}
	\D u +h(x)e^u = f(x) \quad \mbox{in } \O,
\end{equation}
where $\O\subset\R^2$ is a smooth bounded domain and $h(x)$ is a positive function. 
The latter equation (and its counterpart on manifolds, see \eqref{eq:mf} below) has been widely discussed in
the last decades since it arises in several problems of mathematics and physics, 
such as Electroweak and Chern-Simons self-dual vortices \cite{sy2, T0, yang}, conformal geometry on surfaces \cite{Troy, KW, CY1, CY2}, 
statistical mechanics of two-dimensional turbulence  \cite{clmp2} and of self-gravitating systems \cite{w} and cosmic strings \cite{pot}, 
theory of hyperelliptic curves \cite{cLin14}, Painlev\'e equations \cite{CKLin} and Moser-Trudinger inequalities 
\cite{beck, dolb-est-jan, gho-lin, gui-wei, moser}. There are by now many results concerning  existence and multiplicity 
\cite{B5, BDeM, BdM2, BdMM, BJLY, BMal, bt, cama, cl2, cl4, dem2, DJLW, dj, linwang, Mal1, Mal2}, 
uniqueness \cite{bl, BLin3, BLT, GM1, GM3, Lin1, Lin7, suz}, blow-up phenomena \cite{bcct, BJLY2, bt2, bm, cl1, cl3, yy, ls, Tar14, Za2} and 
entire solutions \cite{barjga, cli1, PT}.

\medskip

\subsection{Singular Sphere Covering Inequality} A basic inequality related to \eqref{eq:liouv} 
was recently introduced in \cite{GM1}, see Theorem A below, which yields sharp Moser-Trudinger inequalities and symmetry 
properties of Liouville type equations in $\R^2$ \cite{GM1}, symmetry properties of mean field equations on flat tori \cite{GM2}, 
uniqueness of solutions of mean field equation in bounded domains \cite{GM3} and symmetry and uniqueness properties of 
Sinh-Gordon equation and Toda systems in bounded domains \cite{gui-jev-mor}. The inequality can be stated in the following form,
see Theorem 3.1 in \cite{GM1} and Theorem 1.1 in \cite{GM3}.

\medskip

\noindent \textbf{Theorem A} (\cite{GM1})\textbf{.} \emph{Let $\O\subset\R^2$ 
be a smooth, bounded, simply-connected domain and let $u_i\in C^2(\O)\cap C(\ov\O)$, $i=1,2$, satisfy,
\begin{equation} \label{eq}
	\D u_i + h(x)e^{u_i}=f_i(x) \mbox{ in } \O,
\end{equation}
where $h=e^{H}\in C^2(\O)\cap C(\ov\O)$ is such that,
\begin{equation} \label{subharmonic}
	f_2\geq f_1\geq -\D H  \mbox{ in } \O.
\end{equation}  
Suppose that,
$$	
\left\{ \begin{array}{ll}
u_2\geq u_1, \quad u_2 \not\equiv u_1 & \mbox{ in } \O, \vspace{0.2cm}\\
u_2=u_1 & \mbox{ on } \p \O.
\end{array}
\right.
$$ 
Then it holds,
$$
	\int_\O \left( h(x)e^{u_1}+h(x)e^{u_2} \right)\,dx \geq 8\pi.
$$}

\medskip

The latter result is based on symmetric rearrangements and the Alexandrov-Bol inequality, see the
discussion in the sequel. Suppose for simplicity $h\equiv 1$ and $f_1\equiv f_2 \equiv 0$ in \eqref{eq}.
Then, by means of the substitution $\wtilde u_i = \dfrac{u_i}{2}-\ln(2)$, 
Theorem A roughly asserts that the total area of two distinct neighbourhoods $M_1,M_2$, with Gaussian curvature equal to $1$, 
of possibly distinct surfaces, such that $M_1$ and $M_2$ admits local conformal charts $\Phi_i: M_i\to B_1$, $i=1,2$ where $B_1$ is the 
Euclidean unit disk, with the same conformal factor on the boundary, is greater than that of the whole unit sphere, which is why one refers to the latter result as the Sphere Covering Inequality. 

\medskip

We point out that condition \eqref{subharmonic} on the weight $h(x)$ can not be dropped and indeed the result 
is false in general if we remove such assumption. Suppose for a moment $f_1\equiv f_2\equiv 0$ in Theorem A. 
With a small abuse of terminology, a weight $h$ satisfying \eqref{subharmonic} with $f_1 \equiv f_2\equiv 0$, i.e. $\D H\geq 0$, 
will be referred as a subharmonic weight. Analogously, we will refer to superharmonic weights whenever we have the reverse inequality $\D H\leq 0$. 
Such restriction on the weights, which in this case need to be subharmonic, prevents the application of the Sphere Covering Inequality 
in a large class of problems which are rather interesting and 
challenging, some of which we will address later on.

\medskip

On the other hand, Theorem A is obtained in a smooth setting and thus is not suitable in treating singular problems.
However, the presence of singular sources in \eqref{eq:liouv} naturally arises both in geometry and in mathematical physics, 
typically as a sum of Dirac deltas which represent either conic points of the metric or  
vortex points of gauge or vorticity fields (see the references above). One of our aims is to address such kind of problems, see subsections \ref{subsec:mf}, \ref{subsec:domain}. However, the generalization of Theorem~A to a weak setting is not straightforward and one needs to carry out a delicate argument.

\  

Therefore, our first goal in this paper is to provide a singular version of the Sphere Covering Inequality suitable for treating both singular problems and superharmonic weights, see Theorem~\ref{thm:ineq} below. In order to state the result let us introduce some notation first. Given $\a\in[0,1)$, $\l>0$ we set
\begin{equation} \label{U}
	U_{\l,\a}(x) = \ln\left( \dfrac{ \l(1-\a) }{ 1+\frac{\l^2}{8}|x|^{2(1-\a)}} \right)^2,
\end{equation}
which satisfies
\begin{equation} \label{eq:U}
	\D U_{\l,\a} + |x|^{-2\a}e^{U_{\l,\a}}=0 \mbox{ in }\R^2\setminus \{0\}.
\end{equation}

\medskip

Let $\O\subset\R^2$ be a smooth and bounded domain. Let $f\in L^{q}(\O)$ for some $q>1$ and $h=e^H$ be given. Let $F$ be the solution of 
$$
	\D F= f  \mbox{ in $\O$},\quad F=0 \mbox{ on $\partial \O$}, 
$$
so that \eqref{eq:liouv} can be equivalently formulated as
$$
	\D v+e^{H+F}e^v=0 \mbox{ in } \O,
$$
where $v=u-F$. By the Riesz decomposition we have,
$$
	H+F=\mathcal H_+-\mathcal H_-,
$$
where $\mathcal H_+, \mathcal H_-$ are two superharmonic functions taking the form
\begin{equation} \label{om+}
	\mathcal H_{\pm}(x) = \mathfrak h_{\pm}(x) + \int_\O G_x(y)\,d\mu_{\pm}(y),
\end{equation}
where $\mathfrak h_{\pm}$ are harmonic functions in $\O$, $\mu_{\pm}$ are non negative and mutually orthogonal 
measures of bounded total variation in $\O$ and $G_p$, $p\in\O$, is the Green's function,
\begin{equation} \label{green}
	\left\{ \begin{array}{rll}
						-\D G_p(y)=&\d_p & \mbox{in } \O, \\
						G_p(y)=&0 & \mbox{on } \p\O.
					\end{array}
					\right.
\end{equation} 
\begin{definition} \label{def1}
Let $\o\subseteq\O$ be any nonempty subdomain. We will denote by $\wtilde\o$ the interior of the closure of the union of $\o$ with its 
``holes", that is, with the bounded component of the complement of $\o$ in $\R^2$. 
If $\o$ is simply-connected or if it is the union of simply connected domains, then $\wtilde \o\equiv \o$, while if 
the bounded component of the complement of $\o$ in $\R^2$ is not empty, then $\o\subset \wtilde \o$.
\end{definition}

\begin{definition} 
Let $f, h$ be given as above and let $\mu_+$ be defined as in \eqref{om+}. Let $\o\subseteq\O$ be any nonempty subdomain and let $\wtilde\o$ be given as in Definition \ref{def1}. We define $\a(\o)=\a(\o,h,f)\geq 0$ to be
\begin{equation} \label{a}
	\a(\o)=\dfrac{1}{4\pi}\,\mu_+(\wtilde \o).
\end{equation}
\end{definition}

\medskip
								
For reader's convenience let us show how the above construction applies to two model cases.

\medskip

\begin{example} \label{ex1}
\emph{Let us consider a smooth setting where $f, h\in C^2(\O)\cap C(\ov\O)$ and suppose for simplicity $\wtilde\o=\o$. Then we have,
$$
	\a(\o)=\dfrac{1}{4\pi}\int_{E_+} -\D\bigr( H(x)+F(x)\bigr)\,dx=\dfrac{1}{4\pi}\int_{E_+} -\bigr( \D H(x)+f(x)\bigr)\,dx,
$$	
where	
$$
	E_+=\bigr\{x\in\o \, : \, -\bigr( \D H(x)+f(x)\bigr)>0 \bigr\},
$$
while clearly $\a(\o)=0$ whenever $E_+= \emptyset$.
}
\end{example}

\medskip

\begin{example}
\emph{Let us consider a singular weight $h=e^H$ in the form
$$
	H(x)= 4\pi \sum_{i=1}^m \a_i G_{p_i}(x)-4\pi\sum_{j=1}^n \b_j G_{q_j}(x),
$$
where $G_p(x)$ is given in \eqref{green}, $p_i, q_j$ are distinct points in $\O$ and $\a_i, \b_j> 0$. We have, 
$$
	-\D H= 4\pi \sum_{i=1}^m \a_i\d_{p_i} -4\pi\sum_{j=1}^n \b_j\d_{q_j}.
$$
Suppose for simplicity $f=0$ and $\wtilde\o=\o$. Then we have $\mu_+=4\pi \sum_{i=1}^m \a_i \d_{p_i}$ and
$$
	\a(\o)=\sum_{i\in I} \a_i,
$$
where $I=\bigr\{i\in\{1,\dots,m\}\,:\, p_i\in \o\bigr\}$.
}
\end{example}

\medskip

Finally, from now on we will suppose that, 
\begin{equation}\label{hyp1}
 \a(\O)<1\;\;\mbox{i.e.}\;\; \mu_+(\O)<4\pi.
\end{equation}
\noindent

This is motivated by the case when all the measure $\mu_+$ is collapsed to a singular 
Dirac delta, i.e. $\mu_+=4\pi \a_p\d_p$, where we need $\a(\O)=\a_p<1$ to ensure the integrability of $e^{H+F}$. From the geometric point of view this
means that $p$ is not a "cusp" but just a conical singularity, see \cite{Res2}. 

\begin{rem}\label{remxz}
In particular, since we are assuming $\mu_+(\O)<4\pi$, then we deduce the following regularity property. 
Observe that there exists at most one point $x_0\in\O$ such that $\mu_+(x_0)\geq 2\pi$. Thus, if we start from $he^u\in L^1(\O)$ where $u\in L^1(\O)$ is a 
solution of \eqref{eq:liouv} in the sense of distributions, then it can be shown that 
$u\in W^{2,q}(\O)$ for some $q>1$ and that for each $r>0$ small enough there exists $s_r>2$ such that
$u\in W^{2,s_r}(\O\setminus B_r(x_0))$, see \cite{bc}. 
We will quote this regularity property by saying that 
$u\in W_{\mbox{\rm \scriptsize loc}}^{2,s,\mbox{\rm \scriptsize loc}}(\O\setminus\{x_0\})$ for some $s>2$.
In particular we conclude that $u$ is a strong solution of \eqref{eq:liouv}. 
\end{rem}

Our first main result is the following singular Sphere Covering Inequality.
\begin{thm}\label{thm:ineq}
Let $\O\subset\R^2$ be a smooth, bounded, simply-connected domain. Let $x_0\in\O$ be fixed as in Remark \ref{remxz} and
$u_i\in W_{\mbox{\rm \scriptsize loc}}^{2,s,\mbox{\rm \scriptsize loc}}(\O\setminus\{x_0\})\cap W^{2,q}(\O)\cap C(\ov\O)$ for some $s>2$ and some $q>1$, $i=1,2$, satisfy,
\begin{equation} \label{eq2}
	\D u_i + h(x)e^{u_i}=f_i(x) \mbox{ in } \O,
\end{equation}
where $h=e^H$ and $f_1, f_2 \in L^{q}(\O)$ are such that $\a(\O,h,f_i)<1$, $i=1,2$, with $\a(\O,h,f_i)$ defined as in \eqref{a}. Suppose that, 
$$
f_2\geq f_1 \;\mbox{a.e. in } \O,
$$ 
and that there exists a smooth subdomain $\o\subseteq \O$ such that,
\begin{equation} \label{b-cond}	
\left\{ \begin{array}{ll}
u_2 \geq u_1, \quad u_2 \not\equiv u_1 & \mbox{in } \o, \vspace{0.2cm}\\
u_2=u_1 & \mbox{on } \p \o.
\end{array}
\right.
\end{equation} 
Let $\a(\o)=\a(\o,h,f_1)$. Then, it holds
\begin{equation} \label{ineq}
	\int_\o \bigr( h(x)e^{u_1}+h(x)e^{u_2} \bigr)\,dx \geq 8\pi(1-\a(\o)).
\end{equation}
Moreover, the equality holds if and only if (modulo conformal transformations)  
$\o=B_\d(0)$ for some $\d>0$, $f_1\equiv f_2 :=f$, $h(x)e^{u_i}\equiv |x|^{-2\a}e^{U_{\l_i,\a}}$, $i=1,2$, for some $\l_2>\l_1$ 
where $U_{\l_i,\a}$ are defined as in \eqref{U}, $\mu_+=-\Delta H -f=4\pi \a \delta_{p=0}$ in $\o$ and $\a=\a(\o)$.
\end{thm}

\medskip

\begin{rem} \label{f=0}
We point out that if $f_1\equiv f_2$, the condition $u_2 \geq u_1,  u_2 \not\equiv u_1$ in the above theorem can be replaced just by $u_2 \not\equiv u_1$.\\
We also observe that if $\a(\o)=0$ we recover the standard Sphere Covering Inequality of Theorem~A, which however holds now in a weak setting. 
\end{rem}

\begin{rem} \label{smooth}
The smoothness assumption in Theorem \ref{thm:ineq} about 
$\o$ is not necessary and can be remarkably weakened as far as $\o \Subset \Omega$, see \cite{bc}. 
\end{rem}

\medskip

The latter result is obtained in the spirit of Theorem~A with a non trivial adaptation to the weak setting by exploiting weighted 
symmetric rearrangements and a singular Alexandrov-Bol's inequality, see section~\ref{sec:ineq} for full details. 
We stress that, according to the terminology introduced right after Theorem~A, we are now able to cover the case of superharmonic weights.

Suppose for the moment that $f_1\equiv f_2 \equiv 0$ in Theorem~\ref{thm:ineq}. Observe that the case 
$h(x)e^{u_i}\equiv |x|^{-2\a}e^{U_{\l_i,\a}}$, $i=1,2$, corresponding to the equality in Theorem~\ref{thm:ineq}
is a limiting case in which all the measure $\mu_+$ is concentrated in a Dirac delta, i.e. $\mu_+=4\pi \a\d_{p=0}$. In particular, 
the latter Dirac delta corresponds to a conical singularity of 
order $-\a$ on the surface related to \eqref{eq:U}, see for example \cite{bc}. 
In other words, by means of the substitution $\wtilde u_i = \dfrac{u_i}{2}-\ln(2)$, in the terminology of singular surfaces (in the sense of Alexandrov \cite{Res2}), 
Theorem \ref{thm:ineq} roughly asserts that
 the total area of two distinct neighbourhoods $M_1,M_2$, 
with regular Gaussian curvature equal to $1$, of possibly distinct singular  surfaces, 
 such that $M_1$ and $M_2$ admits local conformal charts 
$\Phi_i: M_i\to B_1$, $i=1,2$ where $B_1$ is the 
Euclidean unit disk, with the same conformal factor on the boundary, and the same local total (singular) curvature (which is $\frac{1}{4\pi}\mu=-\frac{1}{4\pi} \Delta H$ in $B_1$), 
is greater than that of a whole unit sphere with two antipodal conical singularities 
of order $-\a$ (where $\a$ is $\frac{1}{4\pi}\mu_+(B_1)$), namely an 'American football'. We refer to \cite{bc} for more details concerning this geometric interpretation.
Therefore we allude to Theorem \ref{thm:ineq} as the singular Sphere Covering Inequality.

\

Let us return to equation \eqref{eq2}. If we further assume the two solutions $u_1, u_2$ have the same total mass in $\O$, i.e. that \eqref{uguale} below holds, then 
we can argue as in \cite{GM3} and improve Theorem \ref{thm:ineq}. More precisely, we can relax the boundary condition in \eqref{b-cond} and treat more general situations, see \eqref{b-cond2}. This will be crucially used in proving uniqueness of solutions of singular Liouville equation \eqref{w} on bounded domains, see Theorem \ref{thm:domain}. Even though Theorem \ref{thm:ineq2} shares some similarities with Theorem \ref{thm:ineq}, its proof, based on a reversed Alexandrov-Bol's inequality, substantially differs from the proof of Theorem \ref{thm:ineq}, see section~\ref{sec:domain}. Our second main result is the following.

\begin{thm} \label{thm:ineq2}
Let $\O\subset\R^2$ be a smooth, bounded, simply-connected domain. Let $x_0\in\O$ be fixed as in Remark \ref{remxz} and
$u_i\in W_{\mbox{\rm \scriptsize loc}}^{2,s,\mbox{\rm \scriptsize loc}}(\O\setminus\{x_0\})\cap W^{2,q}(\O)\cap C(\ov\O)$ for some $s>2$ and some $q>1$, $i=1,2$, satisfy,
$$
	\D u_i + h(x)e^{u_i}=f_i(x) \mbox{ in } \O,
$$
where $h=e^H$ and $f_1, f_2 \in L^{q}(\O)$ are such that $\a(\O,h,f_i)<1$, $i=1,2$, where $\a(\O,h,f_i)$ are defined in \eqref{a}. Suppose that,
$$
f_2\geq f_1 \; \mbox{a.e. in } \O,
$$ 
and that,
\begin{equation} \label{b-cond2}	
\left\{ \begin{array}{ll}
u_1 \not\equiv u_2 & \mbox{in } \O, \vspace{0.2cm}\\
u_2-u_1=c & \mbox{on } \p \O,
\end{array}
\right.
\end{equation} 
for some $c\in\R$. Suppose moreover that,
\begin{equation} \label{uguale}
	\int_\O h(x)e^{u_1}\,dx=\int_\O h(x)e^{u_2} \,dx =\rho,
\end{equation}
and set $\a(\O)=\a(\O,h,f_1)$. Then, it holds $\rho>8\pi(1-\a(\O))$.
\end{thm} 

\medskip

It is worth to point out the following fact concerning the assumption about $\O$ being simply-connected.

\begin{rem} \label{rem:non-sc}
Actually, both Theorems \ref{thm:ineq}, \ref{thm:ineq2} hold for multiply-connected domains $\O$ provided 
the solutions of \eqref{eq2} take constant values on $\p\O$, more precisely $u_i+H=c_i$ on $\p\O$ for some $c_1, c_2 \in\R$. 
This follows from the fact that in the latter situation an Alexandrov-Bol's inequality related to \eqref{eq2} with $\O$ multiply-connected is available, 
see  \cite{BLin3}.
\end{rem}

\

To motivate our studies and to see some applications of Theorems \ref{thm:ineq} and \ref{thm:ineq2} we will address uniqueness issues 
concerning singular Liouville-type equations both on spheres and on bounded domains, 
and symmetry properties for the spherical Onsager vortex equation, see respectively subsections \ref{subsec:mf}, \ref{subsec:domain}, \ref{subsec:onsager} below. 
The argument will yield both to new results as well as to new self-contained proofs of previously known results,
such as the uniqueness of spherical convex polytopes first established in \cite{lt}.
We believe that Theorems \ref{thm:ineq} and \ref{thm:ineq2} will have several 
other applications as it was for Theorem~A (see the discussion before Theorem~A). This will be the topic of our forthcoming papers.

\

\subsection{The singular Liouville equation on $\S^2$.} \label{subsec:mf} Let us start by considering the following equation
\begin{align} \label{eq:mf}
\begin{split}
	& \D_{g} v + \rho \left( \dfrac{e^v}{\int_{\S^2}e^v\,dV_g} -\frac{1}{4\pi}\right)= 4\pi\sum_{j=1}^N \a_j\left(\d_{p_j}-\frac{1}{4\pi}\right) \quad \mbox{on } \S^2, \\
	& \int_{\S^2} v\,dV_g =0,
\end{split}
\end{align}
where $\rho$ is a positive parameter, $\{p_1,\dots,p_N\}\subset\S^2$, $\a_j>-1$ for $j=1,\dots,N$ and $\S^2\subset\R^3$ is the unit sphere with $|\S^2|=4\pi$ 
equipped with its standard Riemannian metric $g$, $\D_g$ is the Laplace-Beltrami operator and $dV_g$ is the volume form. 
Since the equation in \eqref{eq:mf} is invariant under translations $v\mapsto v+c$, then we can normalize solutions to have zero mean value.\\ 
Problem \eqref{eq:mf} is related to both mean field equations with vortex points and spherical metrics with conic singularities. 
We refer to the references in the beginning of the introduction for more details and some of the known results. 
\medskip

Equation \eqref{eq:mf} can be equivalently considered on the plane $\R^2$ via the stereographic projection: suppose without loss of generality that no one of the points $p_i$'s coincides with the north pole $\mathcal N=(0,0,1)\in\R^3$ and let $\Pi:\S^2\setminus\{\mathcal N\}\to \R^2$ be the stereographic projection with respect to $\mathcal N$, i.e.
\begin{equation} \label{Pi}
	\Pi(x_1,x_2,x_3)=\left( \frac{x_1}{1-x_3}\,, \frac{x_2}{1-x_3} \right),
\end{equation} 
and define
$$
	w(x)=v(\Pi^{-1}(x))-\ln\left(\int_{\S^2}e^v\,dV_g\right) \quad x\in\R^2.
$$
With a small abuse of notation we will write $\sum_{j} \a_j$ to denote $\sum_{j=1}^N \a_j$. Then, $w$ satisfies
$$
	\D w +\frac{4\rho}{(1+|x|^2)^2}e^w=\frac{\rho-4\pi\sum_j \a_j}{4\pi}\frac{4}{(1+|x|^2)^2}+4\pi\sum_{j=1}^N \a_j \d_{q_j} \quad \mbox{on } \R^2,
$$
for some $\{q_1,\dots,q_N\}\subset\R^2$. Letting further,
$$
	u(x)=w(x)-\frac{\rho-4\pi\sum_j \a_j}{4\pi}\ln(1+|x|^2)+\ln(4\rho),
$$
we get
\begin{equation} \label{eq:R2}
	\D u +h(x)e^u=4\pi\sum_{j=1}^N \a_j \d_{q_j}  \quad \mbox{on } \R^2,
\end{equation}
where
\begin{equation} \label{h}
	h(x)= (1+|x|^2)^{-l}, \qquad l=\frac{4\pi\bigr(2+\sum_j \a_j\bigr)-\rho}{4\pi}\, ,
\end{equation}
and
$$
	\int_{\R^2} h(x)e^u\,dx=\rho.
$$
Next, let us discuss what is the measure $\mu_+$ defined in \eqref{om+} corresponding to equation \eqref{eq:R2}. To make the presentation simpler let us consider the regular case, i.e. $N=0$, postponing the general case to section~\ref{sec:mf}. Recall the notation $h=e^H$. In order to apply either Theorem A or Theorem~\ref{thm:ineq} we need first to consider the sign of
\begin{equation} \label{sign}
	\D H(x)= \frac{\rho-8\pi}{4\pi}\frac{4}{(1+|x|^2)^2}\,.
\end{equation}
Observe that the sign of this term depends on whether $\rho<8\pi$ or $\rho\geq 8\pi$. 
The value $\rho=8\pi$ plays an important role in Liouville-type problems alike \eqref{eq:mf}: for example, 
it is related to the sharp Moser-Trudinger inequality which yields boundedness from below and coercivity
of the energy functional associated to \eqref{eq:mf} for $\rho<8\pi$ (we refer to the survey \cite{Mal1} 
for full details on this matter). For the latter range of the parameter $\rho$ one expects the solution of
\eqref{eq:mf} to be unique: this indeed holds true and it was first obtained in \cite{Lin0} and \cite{Lin1}. 
The argument used is mainly based on the deep results derived in \cite{cheng-lin}: 
by the asymptotic behavior of solutions to equations alike \eqref{eq:R2} one carries 
out the moving plane method to show that all the solutions to \eqref{eq:R2}, \eqref{h} with $\rho<8\pi$ 
are radially symmetric with respect to the origin.

\medskip

However Theorem A does not apply in this framework 
since the term in \eqref{sign} is negative and thus the weight $h(x)$ is superharmonic according to the terminology
introduced right after Theorem A. 

\medskip

The argument in \cite{Lin0} applies also to the singular case  $N=1$ with singular source $4\pi\a_q\d_q$, $\a_q>-1$. 
Then, letting $\a_-=\min\{\a_q,0\}$ one deduces uniqueness of solution for $\rho<8\pi(1+\a_-)$. On the other hand, the case of multiple 
singular sources $N\geq 3$ is almost completely open due to the fact that we can not rely on radial properties any more. 
The only exceptions concern metrics of constant Gaussian curvature with $N$-conical singularities at $p_j$ of negative order $\a_j\in(-1,0)$ on $\S^2$ 
(convex polytopes), that is, solutions of \eqref{eq:mf} with,
$$
\rho=4\pi\bigr( 2+\sum_j \a_j \bigr).
$$
Based on an algebraic geometric approach, the authors in \cite{lt} showed (among other things) the uniqueness of such metrics 
in the subcritcal case \cite{Troy} (i.e. $4\pi\bigr( 2+\sum_j \a_j \bigr)<8\pi\bigr(1+\min\limits_{j}\{\a_j,0\}\bigr)$), when $\a_j\in(-1,0)$ for all $j=1,\dots,N$ and $N\geq 3$.
Here we will exploit the singular Sphere Covering Inequality of Theorem~\ref{thm:ineq} 
to handle both superharmonic weights and the multiple singular sources in \eqref{eq:R2}, see section \ref{sec:mf}. As a consequence we obtain uniqueness 
results relevant for both the mean field theory of 2D turbulence \cite{clmp2, CK} and the uniqueness of convex polytopes \cite{lt}. Our third main result is the following,
\begin{thm} \label{thm:mf} 
Let $\rho>0$ and $\a_j\in(-1,0)$ for $j=1,\dots,N$, $N\geq0$. Then we have:

\medskip

\mbox{$(i)$} If $N\geq0$ and $\rho<4\pi\bigr( 2+\sum_j \a_j \bigr)$, then \eqref{eq:mf} admits at most one solution;

\smallskip

\mbox{$(ii)$} If $N\geq3$ and $\rho=4\pi\bigr( 2+\sum_j \a_j \bigr)$, then \eqref{eq:mf} admits at most one solution.
\end{thm}

\medskip

The proof of Theorem \ref{thm:mf} is rather delicate since, among other facts (see Remark \ref{sharpD}), it requires a careful use of the 
characterization of the equality sign in the singular Sphere Covering Inequality \eqref{ineq}. After all this is not surprising 
since part $(ii)$ is somehow sharp. Indeed, for $N=1$ solutions to \eqref{eq:mf}  with $\rho=4\pi(2+\a_1)$ do not exist, as 
it is well known that  the 'tear drop' (which is $\S^2$ with one conical singularity) does not admit constant curvature, see for example \cite{barjga, Fang}.
On the other hand, for $N=0$ and $N=2$ (still with $-1<\a_1\leq \a_2<0$ and $\rho=4\pi\bigr( 2+\sum_j \a_j \bigr)$), solutions to \eqref{eq:mf} 
are classified and uniqueness does not hold, see \cite{cli1, PT, troy2}.
In particular, uniqueness fails in general if some $\a_j$ is positive, see for example \cite{lt}.

\smallskip

Moreover, Theorem \ref{thm:mf} also covers most of the previously known results and gives a new self-contained proof of them. 
Indeed, for $N=0$ and $\rho<8\pi$, we get the sharp uniqueness result of \cite{Lin0, Lin1}, for $N=1$ and $\rho<8\pi(1+\a_1)<4\pi(2+\a_1)$, 
we obtain the sharp result of \cite{Lin0}, while part $(ii)$ covers the uniqueness of convex polytopes \cite{lt}.\\
Finally, we remark that whenever the subcriticality condition $\rho<8\pi\bigr(1+\min\limits_{j}\{\a_j,0\}\bigr)$ is also satisfied, 
then  we  have existence \cite{Troy} and uniqueness in $(i)$ and $(ii)$.
The fact that the uniqueness threshold $4\pi\bigr( 2+\sum_j \a_j \bigr)$ may be larger than the subcritical threshold 
seems to suggest another possible application of Theorem~\ref{thm:mf} to the non existence issue for \eqref{eq:mf} in the supercritical regime 
$\rho \in \bigr(8\pi\bigr(1+\min\limits_{j}\{\a_j,0\}\bigr),4\pi\bigr( 2+\sum_j \a_j \bigr)\bigr)$. Indeed, 
the recent evaluation of the topological degree $d_{\rho}$ associated to \eqref{eq:mf} in \cite{cl4} shows that if $N\geq 3$ and $-1<\a_1<\dots<\a_N<0$, then $d_\rho=0$ for 
$\rho \in \bigr(8\pi(1+\a_1),4\pi\bigr( 2+\sum_j \a_j \bigr)\bigr)$. If we knew that any such a solution is non degenerate, then we would conclude by Theorem~\ref{thm:mf} that solutions do not exist in this supercritical region. This motivates the following:\\

{\textbf{Open problem.} Is it true that if $N\geq 3$,\, $-1<\a_1<\dots<\a_N<0$ and $\rho \in \bigr(8\pi(1+\a_1),4\pi\bigr( 2+\sum_j \a_j \bigr)\bigr),$ then \eqref{eq:mf} has no 
solutions?}\\

We point out that some non existence results in this direction were obtained in \cite{barjga, BMal, T1} only for the case $N\leq2$ by using Pohozaev-type identities.

\medskip

In concluding this part we mention the nondegeneracy result for solutions of \eqref{eq:mf} for $N\geq 3$ with only one or two negative $\a_j$ recently obtained in \cite{wz}.

\

\subsection{The singular Liouville equation on bounded domains.} \label{subsec:domain} Next let us consider the counterpart of \eqref{eq:mf} on bounded domains, that is,
\begin{equation} \label{eq:domain}
\left\{ \begin{array}{ll}
\D u+\rho \dfrac{e^u}{\int_{\O}e^u\,dx} = 4\pi\dis{\sum_{j=1}^N} \a_j\d_{p_j} & \mbox{in } \O, \vspace{0.2cm}\\
u=0 & \mbox{on } \p \O,
\end{array}
\right.
\end{equation}
where $\O\subset\R^2$ is a smooth open bounded domain, $\rho$ is a positive parameter, $\{p_1,\dots,p_N\}\subset\O$ and $\a_j>-1$ for $j=1,\dots,N$. 
The latter equation is related to mean field equations of turbulent Euler flows and we refer to the references above for more details about this point.

\medskip

As for \eqref{eq:mf} in subsection \ref{subsec:mf} one expects uniqueness of solutions to \eqref{eq:domain} below a certain level of $\rho$. 
Indeed, for the regular case (i.e. $N=0$), uniqueness was first proved in \cite{suz} for $\O$ simply-connected and $\rho<8\pi$, 
then improved in \cite{CCL} for $\rho=8\pi$ and finally generalized to the case of $\O$ multiply-connected in \cite{BLin3} and \cite{GM3} 
for the case of more general boundary conditions.
The argument is mainly based on the Alexandrov-Bol inequality and the study of the linearized equation, 
and it was generalized in \cite{bl} to cover the singular case where $\rho\leq 8\pi$ and $\a_j>0$ for all $j=1,\dots,N$, $N>0$. 
More recently, in \cite{wz} the authors considered the case of one negative singularity, i.e. $\a_1\in(-1,0)$ and $\a_j>0$ for all $j=2,\dots,N$, 
proving uniqueness of solutions provided $\rho\leq 8\pi(1+\a_1)$.

\medskip

On the other hand, the case of multiple negative singular sources is completely open and we will apply the 
singular Sphere Covering Inequality of Theorem \ref{thm:ineq2} to handle this situation. Indeed, suppose that there exist $u_1, u_2$ satisfying \eqref{eq:domain}. We set,
$$
	w_i(x)=u_i(x)-\ln\left( \int_{\O} e^{u_i}\,dx \right)+\ln(\rho), \quad i=1,2,
$$
so that, for $i=1,2$, we have,
\begin{equation} \label{w}
\left\{ \begin{array}{ll}
\D w_i+e^{w_i} = 4\pi\dis{\sum_{j=1}^N} \a_j\d_{p_j} & \mbox{in } \O, \vspace{0.2cm}\\
w_i=c_i & \mbox{on } \p \O,
\end{array}
\right.
\end{equation}
where
$$
	c_i=-\ln\left( \int_{\O} e^{u_i}\,dx \right)+\ln(\rho), \quad i=1,2,
$$
and
$$
	\int_{\O} e^{w_1}\,dx =\int_{\O} e^{w_2}\,dx=\rho. 
$$
Next, let us define,
$$
J=\bigr\{ j\in\{1,\dots,N\}\, : \, \a_j\in(-1,0) \bigr\},
$$
and let $\a=\a(\O,h,f)>0$, satisfying \eqref{hyp1}, be the measure associated to \eqref{w} as defined in \eqref{a}, that is 
$$
	\a=-\sum_{j\in J}\a_j.
$$

Then, Theorem \ref{thm:ineq2} readily yields our fourth main result.
\begin{thm} \label{thm:domain} 
\sloppy{Let $\O\subset\R^2$ be a smooth, bounded and simply-connected domain and fix \mbox{$\rho\leq8\pi(1-\a)$}. 
Then \eqref{eq:domain} admits at most one solution.}
\end{thm}

\medskip

We remark that \eqref{eq:domain} admits a variational formulation and that the corresponding functional is well known to 
be coercive for $\rho<8\pi\bigr(1+\min\limits_j\{\a_j,0\}\bigr)$. Therefore 
we have existence and uniqueness in Theorem \ref{thm:domain}  if either $|J|\geq 2$ and $\rho\leq8\pi(1-\a)$ or if $|J|=1$ 
and $\rho<8\pi(1-\a)\equiv 8\pi(1+\a_1)$ or if $|J|=0$ and $\rho<8\pi$.

\begin{rem}
Since Theorem \ref{thm:ineq2} applies to multiply-connected domains, see Remark \ref{rem:non-sc} for more details, 
we conclude that  Theorem \ref{thm:domain} holds for $\O$ multiply-connected as well. Moreover, we can also treat the case where 
$e^u$ is replaced by $e^H e^u$ with $H$ subharmonic and where $u=f$ on $\p\O$ with $f\in C(\partial \O)$. 
\end{rem}

We point out that Theorem \ref{thm:domain} covers all the previously known results. Indeed, for $N=0$ and $\rho\leq 8\pi$, we get the sharp uniqueness result 
of \cite{CCL} and  \cite{BLin3} (i.e. the regular case). For $N>0$, $\a_j>0$ for all $j=1,\dots,N$ and $\rho\leq 8\pi$, we get the sharp result of \cite{bl}, 
while in the case $N\geq 2$ of only one negative singularity $\a_1\in(-1,0)$ with $\rho\leq 8\pi(1+\a_1)$ we recover the sharp result of \cite{wz}.

\begin{rem}\label{sharpD}
It is worth to make a remark about the discrepancy of the uniqueness thresholds as obtained in Theorems \ref{thm:mf} 
and \ref{thm:domain}. It turns out that a rather elementary but still crucial point in the proof of Theorem \ref{thm:mf} is that, 
since the equation is solved on $\S^2$, then one has an upper bound on the total (positive) singular curvature, see \eqref{sum2}. This estimate, in turn, 
allows one to adopt an optimization trick which rules out the case where the singular Sphere Covering Inequality would yield no information, 
that is when $\a(\o)\geq 1$ in Theorem \ref{thm:ineq}. This is not anymore possible on a bounded domain and we come up with a threshold 
which, for $|J|\geq 2$ and unlike the case of $\S^2$, is \underline{always} lower than the subcritical existence threshold $8\pi\bigr(1+\min\limits_j\{\a_j,0\}\bigr)$.
Actually, it seems that this uniqueness result could have been obtained 
also by an adaptation of the argument in \cite{bl} and it is a challenging open problem to understand whether or not uniqueness still holds 
for $\rho \in\bigr(8\pi(1-\a),8\pi\bigr(1+\min\limits_j\{\a_j,0\}\bigr)\bigr)$ with $|J|\geq 2$.
\end{rem}

\

\subsection{The Onsager mean field equation on the sphere.} \label{subsec:onsager} Let us consider the equation, 
\begin{align}\label{OnsagerVortexPDE}
\begin{split}
& \Delta_g v(y)+\frac{\exp\bigr(\beta v(y)-\gamma \langle n, y \rangle \bigr)}{\int_{\S^2} \exp\bigr(\beta v(y)-\gamma \langle n, y \rangle\bigr) \,dV_g}-\frac{1}{4\pi}=0 \quad \mbox{on } \S^2, \\
& \int_{\S^2}v \,dV_g=0,
\end{split}
\end{align}
where $\langle \cdot, \cdot \rangle$ is the scalar product in $\R^3$, $n= \vec{n}\in \R^3$ is a unit vector, 
$\beta\geq 0$ and $\gamma \in \R$. Since $\gamma$ can be changed to $-\gamma$ by replacing the north pole with the south pole,  
there is no loss of generality in assuming that $\gamma \ge 0$. Observe that the equation in \eqref{OnsagerVortexPDE} is invariant under 
the addition of a constants, which is why we can impose the condition of zero mean value. Equation \eqref{OnsagerVortexPDE} is the mean 
field equation arising from the spherical Onsager vortex theory, see \cite{CK, MiRo, PD}.  

\medskip

Let us briefly list the known results concerning \eqref{OnsagerVortexPDE}. By a moving plane argument, it is shown in 
\cite{Lin1} that if $\beta< 8\pi$, then for any $\gamma \geq 0$ the equation (\ref{OnsagerVortexPDE}) has a unique solution which is axially 
symmetric with respect to $\vec{n}$. Moreover, the author made the following conjecture. 

\medskip

\noindent
{\bf Conjecture B.} \emph{Let $\gamma> 0$ and  $\beta \leq 16\pi$. Then every solution of (\ref{OnsagerVortexPDE}) is axially symmetric with respect to $\vec{n}$. }

\medskip

In this direction the following results for $\beta > 8\pi$ has been proved in \cite{Lin2}. 

\medskip

\noindent
{\bf Theorem C} (\cite{Lin2})\textbf{.} \emph{For every $\gamma >0$, there exists $\beta_0=\beta_0(\gamma)>8 \pi$ such that, for $8\pi < \beta \leq \beta_0$, any solution of (\ref{OnsagerVortexPDE}) is axially symmetric with respect to $\vec{n}$.} 

\medskip

\noindent
{\bf Theorem D} (\cite{Lin2})\textbf{.} \emph{Let $\{v_i\}_i$ be a sequence of solutions of (\ref{OnsagerVortexPDE}) with $\gamma=0$ and 
$\beta_i \rightarrow 16 \pi$. Suppose that $\lim_{i \rightarrow \infty} \sup_{\S^2} v_i(y)=+\infty$. Then $v_i$ is axially symmetric with respect to some direction $ \vec{n}_i$ in $\R^3$  for $i$ large enough.} 

\medskip

Recently, in \cite{GM1} the authors applied the standard Sphere Covering inequality, see Theorem A, to prove the following result. 

\medskip

\noindent
{\bf Theorem E} (\cite{GM1})\textbf{.}
\emph{Suppose $8 \pi < \beta \leq 16 \pi$ and 
$$
0 \leq \gamma \leq \frac{\beta}{8\pi}-1. 
$$
Then every solution of (\ref{OnsagerVortexPDE}) is axially symmetric with respect to $\vec{n}$.} 

\medskip

The aim here is to use the singular Sphere Covering Inequality, Theorem  \ref{thm:ineq}, to get a new symmetry result. Our fifth main result is the following.
\begin{thm}\label{thm:onsager}
Suppose $8\pi < \beta \leq 16 \pi$ and 
\begin{equation}\label{improvedEstimate}
0 \leq \gamma \leq 3-\frac{\beta}{8\pi}+\sqrt{2\left(3-\frac{\beta}{8\pi}\right)\left(2-\frac{\beta}{8\pi}\right)} .
\end{equation}
Then every solution of (\ref{OnsagerVortexPDE}) is evenly symmetric with respect to a plane passing through the origin and containing the vector $\vec{n}$.
\end{thm}

\medskip

\begin{rem}
We point out that for $8 \pi< \beta \leq 16 \pi$ it holds,
$$
3-\frac{\beta}{8\pi} \geq \frac{\beta}{8\pi}-1,
$$
and thus Theorem \ref{thm:onsager} covers a wider range of parameters compared to Theorem E. 
\end{rem}

\

This paper is organized as follows. In section \ref{sec:ineq} we introduce the argument which yields to the proof 
of the singular Sphere Covering Inequality of Theorem~\ref{thm:ineq}, in section \ref{sec:domain} we deduce the improved version of it under same total mass condition, see Theorem \ref{thm:ineq2}, in section \ref{sec:mf} we prove the uniqueness 
result for the singular Liouville equation on $\S^2$, see Theorem~\ref{thm:mf}, and in section \ref{sec:onsager} we finally derive symmetry of solutions for the spherical Onsager vortex equation, i.e. Theorem \ref{thm:onsager}.

\

\section{The Singular Sphere Covering Inequality.} \label{sec:ineq}

\medskip

In this section we derive the singular Sphere Covering Inequality of Theorem~\ref{thm:ineq}. 
The argument is mainly based on weighted symmetric rearrangements and a singular Alexandrov-Bol's inequality in the spirit of Theorem A.

Let us start by recalling the following version of the Alexandrov-Bol inequality, first proved in the analytical framework 
in \cite{band} and more recently generalized to the weak  setting in \cite{bc2, bc}.
\begin{pro}[\cite{bc2,bc}] \label{bol}
Let $\O\subset\R^2$ be a smooth, bounded, simply-connected domain. Let $x_0\in\O$ be fixed as in Remark \ref{remxz} and let 
$u\in W_{\mbox{\rm \scriptsize loc}}^{2,s,\mbox{\rm \scriptsize loc}}(\O\setminus\{x_0\})\cap W_{\mbox{\rm \scriptsize loc}}^{2,q}(\O)$ for some $s>2$ and some $q>1$, satisfy
$$
	\D u + h(x)e^{u}= f(x)  \mbox{ in } \O,
$$
where $h=e^H$ and $f \in L_{\mbox{\scriptsize \emph{loc}}}^{q}(\O)$ are such that $\a(\O,h,f)$ (as defined in \eqref{a}) satisfies $\a(\O,h,f)<1$.
Let $\omega\subseteq \O$ be a smooth subdomain and let $\a(\o)=\a(\o,h,f)$.  Then it holds,
\begin{equation} \label{bol-ineq}
	\left( \int_{\p\omega}\left(h(x)e^{u}\right)^{\frac 12}\,d\s \right)^2 \geq \frac 12 \left( \int_\omega h(x)e^u\,dx \right)\left( 8\pi(1-\a(\o))-\int_\omega h(x)e^u\,dx \right).
\end{equation}
Moreover, the equality holds if and only if (modulo conformal transformations)  
$\o=B_\d(0)$ for some $\d>0$, $h(x)e^{u}\equiv |x|^{-2\a}e^{U_{\l,\a}}$ for some $\l_2>\l_1$ 
where $U_{\l,\a}$ is defined in \eqref{U}, $\mu_+=-\Delta H -f=4\pi \a \delta_{p=0}$ in $\o$ and $\a=\a(\o)$. In particular, if $\o$ is not simply-connected, then the inequality is always strict.
\end{pro}

\medskip

\begin{rem}\label{uniqueAlB}
The smoothness assumption about $\o$ is not necessary and can be remarkably weakened as far as $\o \Subset \Omega$, see \cite{bc}. Moreover,
we point out that if $\mu_+=0$ then $\a(\o,h,f)=0$ and we recover the standard Alexandrov-Bol inequality.
\end{rem}


\medskip

We will need in the sequel the following counterpart of the singular Alexandrov-Bol
inequality in the radial setting which is derived in the spirit of \cite{GM1, suz}.
\begin{pro} \label{rad}
Let $\a\in[0,1)$, $R>0$ and $\psi\in C(\ov{B_R(0)})\cap W^{1,p}(B_R(0))$ for some $p>2$,  be a strictly decreasing radial function satisfying, 
\begin{equation} \label{est-rad}
	\int_{\p B_r(0)} |\n\psi| \,d\s \leq \int_{B_r(0)} |x|^{-2\a}e^{\psi}\,dx \quad \mbox{for a.e. } r\in (0,R).
\end{equation}
Then,
$$
\left( \int_{\p B_R(0)} \left(|x|^{-2\a}e^{\psi}\right)^{\frac 12}\,d\s \right)^2 \geq \frac 12 
\left( \int_{B_R(0)} |x|^{-2\a}e^{\psi}\,dx \right)\left( 8\pi(1-\a)-\int_{B_R(0)}|x|^{-2\a}e^{\psi}\,dx \right).
$$
Moreover, if $\int_{\p B_r(0)} |\n\psi| \,d\s \not\equiv \int_{B_r(0)} |x|^{-2\a}e^{\psi}\,dx $ in $(0,R)$, then the  inequality is strict.
\end{pro}

\begin{proof}
We start by letting $\b=\psi(R)$ and
$$
	k(t)=\int_{\{\psi>t\}} |x|^{-2\a}e^{\psi}\,dx, \qquad \mu(t)= \int_{\{\psi>t\}} |x|^{-2\a}\,dx, \quad t>\b.
$$
Clearly, $k$ and $\mu$ are absolutely continuous and by using the co-area formula we find that,
\begin{equation} \label{co-area}
	-k'(t)=\int_{\{\psi=t\}} \dfrac{|x|^{-2\a}e^{\psi}}{|\n \psi|}\,d\s = -e^t\mu'(t),
\end{equation}
for a.e. $t>\b$. Therefore, by using \eqref{est-rad} and then the Cauchy-Schwarz inequality we deduce that,
\begin{align*}
	-k(t)k'(t) & = \left(\int_{\{\psi>t\}} |x|^{-2\a}e^{\psi}\,dx\right) \left(\int_{\{\psi=t\}} \dfrac{|x|^{-2\a}e^{\psi}}{|\n \psi|}\,d\s\right) \\
			& \geq \left(\int_{\{\psi=t\}} |\n\psi| \,d\s\right)  \left(\int_{\{\psi=t\}} \dfrac{|x|^{-2\a}e^{\psi}}{|\n \psi|}\,d\s\right)  \\
			& = \left(\int_{\{\psi=t\}} \left(|x|^{-2\a}e^{\psi}\right)^{\frac 12}\,d\s\right)^2 =e^t \left(\int_{\{\psi=t\}} |x|^{-\a}\,d\s\right)^2,
\end{align*}
for a.e. $t>\b$. Moreover we have,
$$
	 \left(\int_{\{\psi=t\}} |x|^{-\a}\,d\s\right)^2= 4\pi(1-\a)\int_{\{\psi>t\}} |x|^{-2\a}\,dx=4\pi(1-\a)\mu(t),
$$
and recalling \eqref{co-area} it follows that,
$$
	\dfrac{d}{dt} \left( e^t\mu(t)-k(t)+\dfrac{1}{8\pi(1-\a)}k^2(t) \right) = e^t\mu(t)+\dfrac{1}{4\pi(1-\a)}k'(t)k(t)\leq 0,
$$
for a.e. $t>\b$. Therefore, by integrating the latter equation, we deduce that,
$$
	\left[ e^t\mu(t)-k(t)+\dfrac{1}{8\pi(1-\a)}k^2(t) \right]_{\b}^{+\infty} = -\left( e^\b\mu(\b)-k(\b)+\dfrac{1}{8\pi(1-\a)}k^2(\b) \right) \leq 0,
$$
namely,
$$
	e^\b\mu(\b) \geq k(\b)\left( 1- \dfrac{1}{8\pi(1-\a)}k(\b)\right).
$$
To conclude the proof  it is enough to observe that,
$$
	k(\b)=\int_{B_R(0)} |x|^{-2\a}e^{\psi}\,dx 
$$
and that,
\begin{align*}
	e^\b\mu(\b)&=e^\b\int_{B_R(0)} |x|^{-2\a}\,dx =e^\b \dfrac{1}{4\pi(1-\a)}\left(\int_{\p B_R(0)} |x|^{-\a}\,d\s\right)^2 \\
		&=\dfrac{1}{4\pi(1-\a)}\left(\int_{\p B_R(0)} \left(|x|^{-2\a}e^{\psi}\right)^{\frac 12}\,d\s\right)^2.
\end{align*}
Furthermore, going back through the argument it is clear that if $\int_{\p B_r(0)} |\n\psi| \,d\s \not\equiv \int_{B_r(0)} |x|^{-2\a}e^{\psi}\,dx $ in $(0,R)$, then the inequality in Proposition \ref{rad} is strict.
\end{proof}

\medskip

With a similar argument it is possible to prove the following reversed Alexandrov-Bol inequality.
\begin{pro} \label{bol-rev}
Let $\a\in[0,1)$, $R>0$ and $\psi\in C(\R^2\setminus B_R(0))\cap W_{\mbox{\scriptsize \emph{loc}}}^{1,p}(\R^2\setminus B_R(0))$ for some $p>2$,  be a strictly decreasing radial function satisfying, 
\begin{equation} \label{est-rev}
	\int_{\p B_r(0)} |\n\psi| \,d\s \leq 8\pi(1-\a)-\int_{\R^2\setminus B_r(0)} |x|^{-2\a}e^{\psi}\,dx \quad \mbox{for a.e. } r\in (R,+\infty)
\end{equation}
and $\int_{\R^2\setminus B_R(0)} |x|^{-2\a}e^{\psi}\,dx<8\pi(1-\a)$. Then,
$$
\left( \int_{\p B_R(0)} \left(|x|^{-2\a}e^{\psi}\right)^{\frac 12}\,d\s \right)^2 \leq \frac 12 
\left( \int_{\R^2\setminus B_R(0)} |x|^{-2\a}e^{\psi}\,dx \right)\left( 8\pi(1-\a)-\int_{\R^2\setminus B_R(0)}|x|^{-2\a}e^{\psi}\,dx \right).
$$
Moreover, if $\int_{\p B_r(0)} |\n\psi| \,d\s \not\equiv 8\pi(1-\a)-\int_{\R^2\setminus B_r(0)} |x|^{-2\a}e^{\psi}\,dx $ in $(R,+\infty)$, then the  inequality is strict.
\end{pro}

\begin{proof}
We let $\b=\psi(R)$ and
$$
	k(t)=8\pi(1-\a)-\int_{\{\psi<t\}} |x|^{-2\a}e^{\psi}\,dx, \quad \mu(t)= \int_{\{\psi>t\}} |x|^{-2\a}\,dx +\dfrac{\pi}{1-\a} R^{2(1-\a)}, 
$$
for $t<\b$. The argument follows then the same steps of the proof of Proposition~\ref{rad} and we refer to \cite{GM2} for further details.
\end{proof}

\medskip

Next, as in \cite{GM1}, we relate the strictly decreasing radial function $\psi$ satisfying \eqref{est-rad} with the functions $U_{\l_1,\a}, U_{\l_2,\a}$ defined in \eqref{U} with $\l_2>\l_1$ and $\a\in[0,1)$, such that $\psi=U_{\l_1,\a}=U_{\l_2,\a}$ on $\p B_R(0)$. 
\begin{lem} \label{comp}
Let $U_{\l_1,\a}, U_{\l_2,\a}$ be defined as in \eqref{U} with $\l_2>\l_1$ and $\a\in[0,1)$. 
Let $\psi\in C(\ov{B_R(0)})\cap W^{1,p}(B_R(0))$ for some $p>2$,  be a strictly decreasing radial function satisfying, 
\begin{equation} \label{estim}
	\int_{\p B_r(0)} |\n\psi| \,d\s \leq \int_{B_r(0)} |x|^{-2\a}e^{\psi}\,dx \quad \mbox{for a.e. } r\in (0,R)
\end{equation}
and $\psi=U_{\l_1,\a}=U_{\l_2,\a}$ on $\p B_R(0)$. Then,  either
\begin{align*}
	 &\int_{B_R(0)} |x|^{-2\a}e^{\psi}\,dx \leq \int_{B_R(0)} |x|^{-2\a}e^{U_{\l_1,\a}}\,dx \\
	 \mbox{or} \qquad &\int_{B_R(0)} |x|^{-2\a}e^{\psi}\,dx \geq \int_{B_R(0)} |x|^{-2\a}e^{U_{\l_2,\a}}\,dx.
\end{align*}	
If $\int_{\p B_r(0)} |\n\psi| \,d\s \not\equiv \int_{B_r(0)} |x|^{-2\a}e^{\psi}\,dx $ in $(0,R)$, then the above inequalities are strict. 

Moreover, we have
$$
\int_{B_R(0)} \left(|x|^{-2\a}e^{U_{\l_1,\a}}+|x|^{-2\a}e^{U_{\l_2,\a}} \right)\,dx = 8\pi(1-\a).
$$
\end{lem}

\begin{proof}
Let us set,
\begin{align*}
	& m_i= \int_{B_R(0)} |x|^{-2\a}e^{U_{\l_i,\a}}\,dx \quad i=1,2, \\
	& m= \int_{B_R(0)} |x|^{-2\a}e^{\psi}\,dx,
\end{align*}
and, recalling that $\psi=U_{\l_1,\a}=U_{\l_2,\a}$ on $\p B_R(0)$,
\begin{align*}
	\b &=\left( \int_{\p B_R(0)} \left(|x|^{-2\a}e^{\psi}\right)^{\frac 12}\,d\s \right)^2 \\
		& =\left( \int_{\p B_R(0)} \left(|x|^{-2\a}e^{U_{\l_1,\a}}\right)^{\frac 12}\,d\s \right)^2 = \left( \int_{\p B_R(0)} \left(|x|^{-2\a}e^{U_{\l_2,\a}}\right)^{\frac 12}\,d\s \right)^2.
\end{align*}
By Proposition \ref{rad} we know that,
\begin{equation} \label{dis-m}
	\b\geq\frac 12 m\bigr(8\pi(1-\a)-m\bigr),
\end{equation}
and by Proposition \ref{bol} and Remark \ref{uniqueAlB} we have,
$$
	\b=\frac 12 m_1\bigr(8\pi(1-\a)-m_1\bigr)=\frac 12 m_2\bigr(8\pi(1-\a)-m_2\bigr).
$$
It follows that $m_1$, $m_2$ are the roots of the following equation:
$$
	y^2-8\pi(1-\a)y+2\b=0.
$$
On the other hand, by \eqref{dis-m} we have,
$$
	m^2-8\pi(1-\a)m+2\b\geq 0.
$$
Then, either $m\leq m_1$ or $m\geq m_2$ which proves 
the alternative of Lemma \ref{comp}. It is clear that if $\int_{\p B_r(0)} |\n\psi| \,d\s \not\equiv \int_{B_r(0)} |x|^{-2\a}e^{\psi}\,dx $ in $(0,R)$, 
then the latter inequalities are strict.
\medskip

We are left with the last equality of Lemma \ref{comp}. This is a standard evaluation and we derive it here for the sake of completeness. First of all 
we have,
\begin{equation} \label{integral U}
	\int_{B_R(0)} |x|^{-2\a}e^{U_{\l_i,\a}}\,dx = 8\pi(1-\a) \dfrac{\l_i^2 R^{2(1-\a)}}{8+\l_i^2 R^{2(1-\a)}}\, , \quad i=1,2.
\end{equation}
Since $U_{\l_1,\a}=U_{\l_2,\a}$ on $\p B_R(0)$ we also find that,
$$
	\dfrac{ \l_1 }{ 1+\frac{\l_1^2}{8}R^{2(1-\a)}}= \dfrac{ \l_2 }{ 1+\frac{\l_2^2}{8}R^{2(1-\a)}}=C,
$$
for some $C>0$. It follows that $\l_1$, $\l_2$ are the roots of the following equation:
\begin{equation} \label{U1}
	y^2-\dfrac{8}{C R^{2(1-\a)}} y+\dfrac{8}{R^{2(1-\a)}}=0, 
\end{equation}
and thus, 
\begin{equation} \label{U2}
\l_1+\l_2=\dfrac{8}{C R^{2(1-\a)}}\,, \qquad \l_1\l_2=\dfrac{8}{R^{2(1-\a)}}\,.
\end{equation}
By using \eqref{U1} and then \eqref{U2} we conclude that,
\begin{align*}
 \int_{B_R(0)} \left(|x|^{-2\a}e^{U_{\l_1,\a}}+|x|^{-2\a}e^{U_{\l_2,\a}} \right)\,dx & = 8\pi(1-\a)\left(\dfrac{\l_1^2 R^{2(1-\a)}}{8+\l_1^2 R^{2(1-\a)}} + \dfrac{\l_2^2 R^{2(1-\a)}}{8+\l_2^2 R^{2(1-\a)}} \right) \\
		& = 8\pi(1-\a)\left(\dfrac{\l_1^2 R^{2(1-\a)}}{ \frac{8\l_1}{C} } + \dfrac{\l_2^2 R^{2(1-\a)}}{ \frac{8\l_2}{C}  } \right) \\
		& = 8\pi(1-\a)\left(\dfrac{C R^{2(1-\a)}}{8}(\l_1+\l_2) \right) \\
		& = 8\pi(1-\a).
\end{align*}
The proof of Lemma \ref{comp} is now complete.
\end{proof}

\medskip

On the other hand, an analogous argument with obvious modifications based on Proposition \ref{bol-rev} yields the following result (see \cite{GM2} for full details).
\begin{lem} \label{comp2}
Let $U_{\l_1,\a}, U_{\l_2,\a}$ be defined as in \eqref{U} with $\l_2>\l_1$ and $\a\in[0,1)$. 
Let $\psi\in C(\R^2\setminus B_R(0))\cap W_{\mbox{\scriptsize \emph{loc}}}^{1,p}(\R^2\setminus B_R(0))$ for some $p>2$,  be a strictly decreasing radial function satisfying, 
$$
	\int_{\p B_r(0)} |\n\psi| \,d\s \leq 8\pi(1-\a)-\int_{\R^2\setminus B_r(0)} |x|^{-2\a}e^{\psi}\,dx \quad \mbox{for a.e. } r\in (R,+\infty)
$$
and $\int_{\R^2\setminus B_R(0)} |x|^{-2\a}e^{\psi}\,dx<8\pi(1-\a)$. Suppose $\psi=U_{\l_1,\a}=U_{\l_2,\a}$ on $\p B_R(0)$. Then,  
$$
	\int_{\R^2\setminus B_R(0)} |x|^{-2\a}e^{U_{\l_2,\a}}\,dx \leq \int_{\R^2\setminus B_R(0)} |x|^{-2\a}e^{\psi}\,dx \leq \int_{\R^2\setminus B_R(0)} |x|^{-2\a}e^{U_{\l_1,\a}}\,dx.
$$	
If $\int_{\p B_r(0)} |\n\psi| \,d\s \not\equiv \int_{\R^2\setminus B_r(0)} |x|^{-2\a}e^{\psi}\,dx $ in $(R,+\infty)$, then the above inequalities are strict.
\end{lem}

\medskip

Finally, let us introduce some known facts about weighted symmetric rearrangements with respect to two measures: 
under the assumptions of Proposition \ref{bol}, for a given function $\phi\in C(\ov{\O})\cap W^{1,p}(\O)$ for some $p>2$, such that $\phi= C$ on $\p\o$, $\o\subseteq\O$, 
we will consider its equimeasurable rearrangement in $\o$ with respect to the measures $h(x)e^u \,dx$ 
and $|x|^{-2\a}e^{U_{\l,\a}}\,dx$, where $\a=\a(\o,h,f)$ and $U_{\l,\a}$ are defined in \eqref{a}, \eqref{U}, and $u, h, f$ 
are given in Proposition \ref{bol}. More exactly, for $t>t_0=\min_{x\in\ov\o}\phi(x)$ let 
\begin{align} \label{sets}
\begin{split}
	& \{\phi=t\} = \bigr\{ x\in\o \,:\, \phi(x)=t \bigr\}\subseteq\o, \\
	& \o_t = \bigr\{ x\in\o \,:\, \phi(x)>t \bigr\}\subseteq\o,
\end{split}	
\end{align}
and let $\mathcal B_t^*$ be the ball centered at the origin such that
$$
	\int_{\mathcal B_t^*} |x|^{-2\a}e^{U_{\l,\a}}\,dx = \int_{\o_t} h(x)e^u\,dx.
$$
Then, $\phi^*:\mathcal B_{t_0}^*\to\R$ defined by $\phi^*(x)=\sup\bigr\{t\in\R\,:\,x\in\mathcal B_t^*\bigr\} $ is a radial, decreasing, 
equimeasurable rearrangement of $\phi$ with respect to the measures $h(x)e^u \,dx$ and $|x|^{-2\a}e^{U_{\l,\a}}\,dx$, i.e. 
$\{\phi^*>t\}\equiv \mathcal B_t^*$ and,
\begin{equation} \label{equi}
	\int_{\{\phi^*>t\}} |x|^{-2\a}e^{U_{\l,\a}}\,dx = \int_{\o_t} h(x)e^u\,dx,
\end{equation}
for all $t>\min_{x\in\ov\o}\phi(x)$. Elementary arguments show that $\phi^*$ is a BV function. By exploiting Proposition~\ref{bol} we get the following property.
\begin{lem} \label{rearr}
Under the assumptions of Proposition \ref{bol}, let $\phi\in C(\ov{\O})\cap W^{1,p}(\O)$ for some $p>2$ be such that $\phi = C$ on $\p\o$, $\o\subseteq\O$. 
If $\phi^*$ is the equimeasurable symmetric rearrangement of $\phi$ in $\o$ with respect to the measures 
$h(x)e^u \,dx$ and $|x|^{-2\a}e^{U_{\l,\a}}\,dx$ defined above with $\a=\a(\o,h,f)$, then
$$
\int_{\{\phi^*=t\}} |\n \phi^*|\,d\s \leq\int_{\{\phi=t\}} |\n \phi|\,d\s,
$$
for a.e. $t>\min_{x\in\ov\o}\phi(x)$.
\end{lem} 

\begin{proof}
The argument is standard so we will be sketchy and refer to \cite{bc, bc2, bl, CCL, GM1} for full details. 
We first apply the Cauchy-Schwarz inequality and then the co-area formula to get that, 
\begin{align*}
	\int_{\{\phi=t\}} |\n \phi|\,d\s & \geq \left( \int_{\{\phi=t\}} \left(h(x)e^{u}\right)^{\frac 12}\,d\s \right)^2 \left( \int_{\{\phi=t\}} \dfrac{h(x)e^u}{|\n \phi|}\,d\s \right)^{-1} \\
	  & = \left( \int_{\{\phi=t\}} \left(h(x)e^{u}\right)^{\frac 12}\,d\s \right)^2 \left( -\dfrac{d}{dt}\int_{\o_t} h(x)e^u\,dx \right)^{-1},
\end{align*}
for a.e. $t$. Then, in view of the Alexandrov-Bol inequality of Proposition \ref{bol} we see that,
\begin{align*}
	&\left( \int_{\{\phi=t\}} \left(h(x)e^{u}\right)^{\frac 12}\,d\s \right)^2 \left( -\dfrac{d}{dt}\int_{\o_t} h(x)e^u\,dx \right)^{-1} \\
	&\geq \frac 12 \left( \int_{\o_t} h(x)e^u\,dx \right)\left( 8\pi(1-\a(\o))-\int_{\o_t} h(x)e^u\,dx \right) \left( -\dfrac{d}{dt}\int_{\o_t} h(x)e^u\,dx \right)^{-1}.
\end{align*}
Since $\phi^*$ is an equimeasurable rearrangement of $\phi$ in $\o$ with respect to the measures $h(x)e^u \,dx$, $|x|^{-2\a}e^{U_{\l,\a}}\,dx$, and since 
$|x|^{-2\a}e^{U_{\l,\a}}$ realizes the equality in Proposition~\ref{bol},
then we can argue in the other way around, 
\begin{align*}
&\frac 12 \left( \int_{\o_t} h(x)e^u\,dx \right)\left( 8\pi(1-\a(\o))-\int_{\o_t} h(x)e^u\,dx \right) \left( -\dfrac{d}{dt}\int_{\o_t} h(x)e^u\,dx \right)^{-1} \\
&=\frac 12 \left( \int_{\mathcal B_t^*} |x|^{-2\a}e^{U_{\l,\a}}\,dx \right)\left( 8\pi(1-\a(\o))-\int_{\mathcal B_t^*} |x|^{-2\a}e^{U_{\l,\a}}\,dx \right) \left( -\dfrac{d}{dt}\int_{\mathcal B_t^*} |x|^{-2\a}e^{U_{\l,\a}}\,dx \right)^{-1} \\
& =\left( \int_{\{\phi^*=t\}} \left(|x|^{-2\a}e^{U_{\l,\a}}\right)^{\frac 12}\,d\s \right)^2 \left( -\dfrac{d}{dt}\int_{\mathcal B_t^*} |x|^{-2\a}e^{U_{\l,\a}}\,dx \right)^{-1} \\
& = \int_{\{\phi^*=t\}} |\n \phi^*|\,d\s,
\end{align*}
as claimed, where in the last equality we used the co-area formula for BV functions, see \cite{fr}. 
\end{proof}

\medskip

At this point we are ready to derive the main result of this section, namely the singular Sphere Covering Inequality.

\medskip

\begin{proof}[Proof of Theorem \ref{thm:ineq}]
Let $u_1, u_2$ and $\o\subseteq \O$ be as in the statement of Theorem~\ref{thm:ineq}. 
Let $\a=\a(\o,h,f_1)$, where $\a(\o,h,f_1)$ is given in \eqref{a}, and let $U_{\l_1,\a}, U_{\l_2,\a}$ 
be as defined in \eqref{U} for some $\l_2>\l_1$.  
Take $\l_1, \l_2$ such that $U_{\l_1,\a}=U_{\l_2,\a}$ on $\p B_1(0)$ and,
$$
	\int_{\o} h(x)e^{u_1}\,dx=\int_{B_1(0)} |x|^{-2\a}e^{U_{\l_1,\a}}\,dx.
$$
Let $\phi=u_2-u_1$ and let us assume without loss of generality that,
$$	
\left\{ \begin{array}{ll}
\phi>0 & \mbox{in } \o, \vspace{0.2cm}\\
\phi=0 & \mbox{on } \p \o.
\end{array}
\right.
$$ 
Since $u_1$ satisfies, 
$$
	\D u_1+h(x)e^{u_1}=f_1(x) \mbox{ in } \O,
$$
\sloppy{and $\O$ is simply-connected, then, as in the discussion after Lemma \ref{comp2},
we can define the radial, decreasing, equimeasurable rearrangement $\phi^*$  of 
$\phi$ in $\o$ with respect to the two measures $h(x)e^{u_1} \,dx$ and $|x|^{-2\a}e^{U_{\l_1,\a}}\,dx$. 
Let $\{\phi=t\}$ and $\o_t \subseteq\o$ be defined as in \eqref{sets}. In particular we have,} 
$$
	\int_{\{\phi^*>t\}} |x|^{-2\a}e^{U_{\l_1,\a}}\,dx = \int_{\o_t} h(x)e^{u_1}\,dx,
$$
for $t \ge 0$. Observe now that due to the assumption $f_2\geq f_1$ a.e. in $\O$, we also have,
\begin{equation} \label{diff}
	\D(u_2-u_1)+h(x)e^{u_2}-h(x)e^{u_1}=f_2(x)-f_1(x)\geq 0   \mbox{ a.e. in } \O.
\end{equation}
We first estimate the gradient of the rearrangement by applying Lemma \ref{rearr} and then use equation \eqref{diff} to obtain,  
\begin{align*}
	\int_{\{ \phi^*=t \}} |\n\phi^*|\,d\s & \leq \int_{\{ \phi=t \}} |\n(u_2-u_1)|\,d\s \\
	& \leq \int_{\o_t} \bigr( h(x)e^{u_2}-h(x)e^{u_1} \bigr)\,dx,
\end{align*}
for a.e. $t>0$. By using the equimeasurablility with the Fubini theorem and then also the equation \eqref{eq:U} satisfied by $U_{\l_1,\a}$, we deduce that,
\begin{align*}
	\int_{\o_t} \bigr( h(x)e^{u_2}-h(x)e^{u_1} \bigr)\,dx & = \int_{\{ \phi^*>t \}} |x|^{-2\a}e^{U_{\l_1,\a}+\phi^*}\,dx - \int_{\{ \phi^*>t \}} |x|^{-2\a}e^{U_{\l_1,\a}}\,dx \\
		& = \int_{\{ \phi^*>t \}} |x|^{-2\a}e^{U_{\l_1,\a}+\phi^*}\,dx - \int_{\{ \phi^*=t \}} |\n{U_{\l_1,\a}}|\,d\s,
\end{align*}
for a.e. $t>0$. Therefore, we conclude that,
$$
	\int_{\{ \phi^*=t \}} |\n\bigr(U_{\l_1,\a}+\phi^*\bigr)|\,d\s \leq\int_{\{ \phi^*>t \}} |x|^{-2\a}e^{U_{\l_1,\a}+\phi^*}\,dx,
$$
for a.e. $t>0$. Let $\psi=U_{\l_1,\a}+\phi^*$. Since $\phi^*$ is decreasing by construction  then $\psi$ is a strictly decreasing function. 
Moreover, the above estimate can be rewritten as,
\begin{equation} \label{estimate}
	\int_{\p B_r(0)} |\n\psi|\,d\s \leq \int_{B_r(0)} |x|^{-2\a}e^{\psi}\,dx \quad \mbox{for a.e. } r>0,
\end{equation}
which shows that $\psi\in W^{1,p}(B_1(0))$ for some $p>2$. 
Furthermore, recalling that $\phi> 0$ in $\o$ we have $\phi^*\geq 0$, $\phi^*\not\equiv 0$, which implies that,
$$
	\int_{B_1(0)} |x|^{-2\a}e^{\psi}\,dx=\int_{B_1(0)} |x|^{-2\a}e^{U_{\l_1,\a}+\phi^*}\,dx > \int_{B_1(0)} |x|^{-2\a}e^{U_{\l_1,\a}}\,dx.
$$
Observing that $\phi=0$ on $\p\o$ we have $\phi^*= 0$ on $\p B_1(0)$ and $\psi=U_{\l_1,\a}=U_{\l_2,\a}$ on $\p B_1(0)$. 
By \eqref{estimate} and the above estimate  we can exploit the alternative of Lemma~\ref{comp} about $\psi$, to obtain,
$$
	\int_{B_1(0)} |x|^{-2\a}e^{U_{\l_1,\a}+\phi^*}\,dx = \int_{B_1(0)} |x|^{-2\a}e^{\psi}\,dx > \int_{B_1(0)} |x|^{-2\a}e^{U_{\l_2,\a}}\,dx.
$$
Therefore, by using Lemma \ref{comp} once more, we conclude that,
\begin{align} \label{int}
\begin{split}
\int_{\o}\bigr( h(x)e^{u_1}+h(x)e^{u_2} \bigr)\,dx & =	\int_{B_1(0)} \left( |x|^{-2\a}e^{U_{\l_1,\a}}+|x|^{-2\a}e^{U_{\l_1,\a}+\phi^*} \right)\,dx \\
& \geq \int_{B_1(0)} \left( |x|^{-2\a}e^{U_{\l_1,\a}}+|x|^{-2\a}e^{U_{\l_2,\a}} \right)\,dx \\
&= 8\pi(1-\a).
\end{split}
\end{align}
Moreover, going back through the proof, we see that the equality 
holds only if we have equality in \eqref{diff}, i.e. $f_1\equiv f_2:=f$, and in \eqref{estimate} for a.e. $r>0$. 
It follows that in Lemma \ref{rearr} we have equality in the estimate of the gradient of the rearrangement 
and hence the equality in the Alexandrov-Bol inequality of Proposition \ref{bol} for $u_1$ in $\O$. 
Therefore, (modulo conformal transformations) $\o=B_\d(0)$ for some $\d>0$, $h(x) e^{u_1}\equiv |x|^{-2\a}e^{U_{\wtilde\l_1,\a}}$ 
for some $\wtilde \l_1>0$  and  
\begin{equation}\label{Dh+f}
\mu_+=-\Delta H(x) -f(x)=4\pi \a \delta_{p=0}  \mbox{ in }B_\d(0),
\end{equation}
where we recall that $h=e^H$. In particular we have 
\begin{equation}\label{same}
U_{\wtilde\l_1,\a}\equiv u_1+2\a\ln|x|+H(x).
\end{equation}
Since $\o=B_\d(0)$, and by setting $w_2=u_2+2\a\ln|x|+H(x)$, then \eqref{eq2} and \eqref{Dh+f} imply that $w_2$ satisfies,
$$
	\D w_2+|x|^{-2\a}e^{w_2}=0  \mbox{ in } B_\d(0).
$$
Since $u_2=u_1$ on $\p B_\d(0)$ by assumption, then \eqref{same} implies that $w_2=U_{\wtilde\l_1,\a}$ on $\p B_\d(0)$. Furthermore, observe that, 
as a consequence of the equality sign in \eqref{int}, we have 
$$
	\int_{B_\d(0)} |x|^{-2\a}e^{w_2}\,dx =	\int_{B_\d(0)} h(x)e^{u_2}\,dx \equiv \int_{B_1(0)} |x|^{-2\a}e^{U_{\l_2,\a}}\,dx,
$$
which is already known by \eqref{integral U}. It is then not difficult to see that $w_2\equiv U_{\wtilde\l_2,\a}$ 
for some $\wtilde \l_2>\wtilde \l_1$, i.e.  $h(x) e^{u_2}\equiv |x|^{-2\a}e^{U_{\wtilde\l_2,\a}}$, as claimed. The proof is complete. 
\end{proof}

\

\section{The Singular Sphere Covering Inequality with same total mass.} \label{sec:domain}

\medskip

In this section we will deduce the improved singular Sphere Covering Inequality of Theorem \ref{thm:ineq2} under the same total mass condition, i.e. \eqref{uguale}. This is done in the spirit of \cite{GM3}.

\medskip
	
First of all, let us derive the following lemma which is a simple consequence of the radial Alexandrov-Bol inequality in Proposition \ref{rad}.
\begin{lem} \label{boundary}
Let $U_{\l,\a}$ be defined as in \eqref{U} with $\l>0$ and $\a\in[0,1)$. 
Let $\psi\in C(\ov{B_R(0)})\cap W^{1,p}(B_R(0))$ for some $p>2$,  be a strictly decreasing radial function satisfying, 
$$
	\int_{\p B_r(0)} |\n\psi| \,d\s \leq \int_{B_r(0)} |x|^{-2\a}e^{\psi}\,dx \quad \mbox{for a.e. } r\in (0,R).
$$
Suppose 
$$
	\int_{B_R(0)} |x|^{-2\a}e^{\psi}\,dx=\int_{B_R(0)} |x|^{-2\a}e^{U_{\l,\a}}\,dx<8\pi(1-\a).
$$

\smallskip

Then $U_{\l,\a}(R)\leq \psi(R)$.
\end{lem}

\begin{proof}
We start by using the radial Alexandrov-Bol inequality in Proposition \ref{rad} with $U_{\l,\a}$, exploit the assumptions of the lemma and then apply Proposition \ref{rad} once more to $\psi$ to deduce
\begin{align*}
	\left( \int_{\p B_R(0)} \left(|x|^{-2\a}e^{U_{\l,\a}}\right)^{\frac 12}\,d\s \right)^2 &= \frac 12 
\left( \int_{B_R(0)} |x|^{-2\a}e^{U_{\l,\a}}\,dx \right)\left( 8\pi(1-\a)-\int_{B_R(0)}|x|^{-2\a}e^{U_{\l,\a}}\,dx \right) \\
																&=\frac 12 
\left( \int_{B_R(0)} |x|^{-2\a}e^{\psi}\,dx \right)\left( 8\pi(1-\a)-\int_{B_R(0)}|x|^{-2\a}e^{\psi}\,dx \right) \\
															 &\leq \left( \int_{\p B_R(0)} \left(|x|^{-2\a}e^{\psi}\right)^{\frac 12}\,d\s \right)^2.
\end{align*}
Therefore, it readily follows that $U_{\l,\a}(R)\leq \psi(R)$.
\end{proof}

\medskip

We will now prove the main result of this section, Theorem \ref{thm:ineq2}.

\medskip

\noindent \begin{proof}[Proof of Theorem \ref{thm:ineq2}]
The idea is to proceed as in the proof of the singular Sphere Covering Inequality of Theorem \ref{thm:ineq} and exploit the extra information \eqref{uguale}.
Let $\a=\a(\O,h,f_1)>0$ be as defined in \eqref{a}. We discuss two cases separately.

\medskip
{\bf CASE 1}. We argue by contradiction and suppose first that $\rho<8\pi(1-\a)$. Take $\l_1>0$ and $R>0$ such that
\begin{equation} \label{constr}
	\int_{\O} h(x)e^{u_1}\,dx=\int_{B_R(0)} |x|^{-2\a}e^{U_{\l_1,\a}}\,dx,
\end{equation}
where $U_{\l_1,\a}$ is defined in \eqref{U}. Let $\phi=u_2-u_1$, which, since $u_1\neq u_2$ in $\O$ and in view of \eqref{uguale}, 
obviously changes sign in $\O$. Since $\O$ is simply-connected, then, as right after Lemma \ref{comp}, 
we can define the radial, decreasing, equimeasurable rearrangement $\phi^*$ of $\phi$ in $\O$ with respect to the two measures $h(x)e^{u_1} \,dx$ and $|x|^{-2\a}e^{U_{\l_1,\a}}\,dx$. Moreover, observe that       
\begin{equation} \label{diff2}
	\D(u_2-u_1)+h(x)e^{u_2}-h(x)e^{u_1}=f_2(x)-f_1(x)\geq 0   \mbox{ a.e. in } \O.
\end{equation}
By arguing exactly as in the proof of Theorem \ref{thm:ineq} we deduce that,
$$
	\int_{\{ \phi^*=t \}} |\n\bigr(U_{\l_1,\a}+\phi^*\bigr)|\,d\s \leq\int_{\{ \phi^*>t \}} |x|^{-2\a}e^{U_{\l_1,\a}+\phi^*}\,dx,
$$
for a.e. $t>0$. Let $\psi=U_{\l_1,\a}+\phi^*$. Since $\phi^*$ is decreasing by construction  then $\psi$ is a strictly decreasing function. 
Moreover, the above estimate can be rewritten as,
\begin{equation} \label{prop1}
	\int_{\p B_r(0)} |\n\psi|\,d\s \leq \int_{B_r(0)} |x|^{-2\a}e^{\psi}\,dx \quad \mbox{for a.e. } r>0.
\end{equation}
On the other hand, by \eqref{constr} and \eqref{uguale} we have,
\begin{align*}
	\int_{B_R(0)} |x|^{-2\a}e^{\psi}\,dx & =\int_{B_R(0)} |x|^{-2\a}e^{U_{\l_1,\a}+\phi^*}\,dx=\int_{\O} h(x)e^{u_2}\,dx \\
					&  =\int_{\O} h(x)e^{u_1}\,dx=\int_{B_R(0)} |x|^{-2\a}e^{U_{\l_1,\a}}\,dx.
\end{align*}
Therefore, we observe that,
\begin{equation} \label{prop2}
	\int_{B_R(0)} |x|^{-2\a}e^{\psi}\,dx=\int_{B_R(0)} |x|^{-2\a}e^{U_{\l_1,\a}}\,dx=\int_{\O} h(x)e^{u_1}\,dx=\rho<8\pi(1-\a),
\end{equation}
by assumption. 

\medskip

Once we have \eqref{prop1} and \eqref{prop2} we may apply Lemma \ref{boundary} to get $U_{\l_1,\a}(R)\leq\psi(R)$. However, since $\phi<0$ on a subset of $\O$ with positive measure, then $\phi^*(R)<0$, and thus,
$$
	\psi(R)=U_{\l_1,\a}(R)+\phi^*(R)<U_{\l_1,\a}(R),
$$
which is the desired contradiction.

\medskip
{\bf CASE 2}. Next we suppose that $\rho=8\pi(1-\a)$. Fix $\l_1>0$ and observe that
$$
	\int_{\O} h(x)e^{u_1}\,dx=\rho=8\pi(1-\a)=\int_{\R^2} |x|^{-2\a}e^{U_{\l_1,\a}}\,dx.
$$
As above we set $\phi=u_2-u_1$ and consider its equimeasurable rearrangement $\phi^*$ with respect to the two measures 
$h(x)e^{u_1} \,dx$ and $|x|^{-2\a}e^{U_{\l_1,\a}}\,dx$. Reasoning as in CASE 1 we deduce that,
\begin{equation} \label{est:int}
	\int_{\p B_r(0)} |\n\psi|\,d\s \leq \int_{B_r(0)} |x|^{-2\a}e^{\psi}\,dx \quad \mbox{for a.e. } r>0,
\end{equation}	
where $\psi=U_{\l_1,\a}+\phi^*$ and
\begin{equation} \label{est:tot}
	\int_{\R^2} |x|^{-2\a}e^{\psi}\,dx =\int_{\R^2} |x|^{-2\a}e^{U_{\l_1,\a}}\,dx=8\pi(1-\a).
\end{equation}
In particular, we also have,
\begin{equation} \label{est:ext}
	\int_{\p B_r(0)} |\n\psi|\,d\s \leq 8\pi(1-\a) - \int_{\R^2\setminus B_r(0)} |x|^{-2\a}e^{\psi}\,dx \quad \mbox{for a.e. } r>R,
\end{equation}	
where $R>0$ is any fixed number. By the estimate \eqref{est:tot} we see there exists $R>0$ such that $\psi(R)=U_{\l_1,\a}(R)$. 
Take now $\l_2\neq\l_1$ such that $U_{\l_2,\a}(R)=U_{\l_1,\a}(R)=\psi(R)$. Since $\phi^*$ is decreasing, 
by the definition of $\psi$ and $R$ we conclude $\psi\leq U_{\l_1,\a}$ in $\R^2\setminus B_R(0)$, $\psi\not\equiv U_{\l_1,\a}$. It follows that,
$$
	\int_{\R^2\setminus B_R(0)} |x|^{-2\a}e^{\psi}\,dx < \int_{\R^2\setminus B_R(0)} |x|^{-2\a}e^{U_{\l_1,\a}}\,dx.
$$
Since \eqref{est:ext} holds true, then we may apply Lemma \ref{comp2} and deduce on one side that $\l_1<\l_2$ and on the other side that,
\begin{equation} \label{ext}
	\int_{\R^2\setminus B_R(0)} |x|^{-2\a}e^{\psi}\,dx \geq \int_{\R^2\setminus B_R(0)} |x|^{-2\a}e^{U_{\l_2,\a}}\,dx.
\end{equation}
On the other hand, we observe that  $\psi\geq U_{\l_1,\a}$ in $B_R(0)$, $\psi\not\equiv U_{\l_1,\a}$. It follows that,
$$
	\int_{B_R(0)} |x|^{-2\a}e^{\psi}\,dx > \int_{B_R(0)} |x|^{-2\a}e^{U_{\l_1,\a}}\,dx,
$$
and hence, since \eqref{est:int} holds true, Lemma \ref{comp} yields,
\begin{equation} \label{inte}
	\int_{B_R(0)} |x|^{-2\a}e^{\psi}\,dx \geq \int_{B_R(0)} |x|^{-2\a}e^{U_{\l_2,\a}}\,dx.
\end{equation}
Finally, exploiting \eqref{est:tot} and summing up \eqref{ext} and \eqref{inte} we end up with,
\begin{align}
	8\pi(1-\a)&=\int_{\R^2} |x|^{-2\a}e^{\psi}\,dx=\int_{B_R(0)} |x|^{-2\a}e^{\psi}\,dx +\int_{\R^2\setminus B_R(0)} |x|^{-2\a}e^{\psi}\,dx \nonumber\\
						& \geq \int_{\R^2} |x|^{-2\a}e^{U_{\l_2,\a}}\,dx = 8\pi(1-\a). \label{est:eq}
\end{align}
In order to obtain the desired contradiction we are left with showing that the latter inequality is strict. 
We first observe the equality holds only if we have equality in \eqref{diff2}, i.e. $f_1\equiv f_2$. 
Next, we may suppose without loss of generality $u_2-u_1=c\geq 0$ on $\p\O$. 
The equality in \eqref{est:eq} holds if and only if, in particular, we have equality in \eqref{est:int}. 
It follows that in Lemma~\ref{rearr} we have equality in the estimate of the gradient of the rearrangement $\phi^*$ of  $\phi=u_2-u_1$ 
and hence the equality in the Alexandrov-Bol inequality of Proposition \ref{bol} in $\O_t = \bigr\{ x\in\O \,:\, \phi(x)>t \bigr\}$ for a.e. $t>\min_{\ov{\O}}\phi$. 
Therefore, $\O_t$ must be simply-connected for a.e. $t>\min_{\ov\O}\phi$. Recalling that $\phi\geq 0$ on $\p\O$, this is impossible since $\O_t$ 
is not simply-connected for $\min_{\ov\O}\phi<t<0$. We conclude the inequality in \eqref{est:eq} is strict and the proof is completed.
\end{proof}

\

\section{Uniqueness of solutions of the singular Liouville equation on $\S^2$.} \label{sec:mf}

\medskip

In this section we show a first application of the singular Sphere Covering Inequality of Theorem \ref{thm:ineq} 
by proving the uniqueness of solutions of the singular Liouville equation \eqref{eq:mf} on $\S^2$, see Theorem~\ref{thm:mf} and the discussion in subsection \ref{subsec:mf}.

\medskip

\noindent \begin{proof}[Proof of Theorem \ref{thm:mf}]
Suppose by contradiction that there exist $v_1, v_2$ satisfying \eqref{eq:mf} such that $v_1 \not\equiv v_2$. 
As discussed in subsection \ref{subsec:mf} we let $\Pi:\S^2\setminus\{\mathcal N\}\to \R^2$ be the stereographic projection 
with respect to the north pole defined in \eqref{Pi} and we define,
$$
	u_i(x)=v_i(\Pi^{-1}(x))-\ln\left(\int_{\S^2}e^{v_i}\,dV_g\right)-\frac{\rho-4\pi\sum_j \a_j}{4\pi}\ln(1+|x|^2)+\ln(4\rho), \quad x\in\R^2, \ i=1,2,
$$
where we recall that $\a_j\in(-1,0)$ for all $j=1,\dots,N$, $N\geq0$ and $\sum_j \a_j$ stands for $\sum_{j=1}^N \a_j$. Then, $u_1, u_2$ satisfy,
\begin{equation} \label{eq-ui}
	\D u_i +h(x)e^{u_i}=4\pi\sum_{j=1}^N \a_j\d_{q_j} \quad \mbox{on } \R^2, \quad i=1,2,
\end{equation}
for some $\{ q_j,\dots,q_N \}\subset\R^2$, where,
\begin{equation} \label{h1}
	h(x)= (1+|x|^2)^{-l}, \quad l=\frac{4\pi\bigr( 2+\sum_j \a_j \bigr)-\rho}{4\pi}\, ,
\end{equation}
and,
\begin{equation} \label{tot}
	\int_{\R^2} h(x)e^{u_i}\,dx=\rho, \quad i=1,2.
\end{equation}
We have $u_1\not\equiv u_2$
and so there must exist at least two disjoint simply-connected regions $\O_i \subset\R^2$, $i=1,2$ (not necessarily bounded) 
such that there exist $\o_i\subseteq\O_i$, $i=1,2$, such that,
$$	
\left\{ \begin{array}{l}
u_1> u_2 \quad  \mbox{in } \o_1, \qquad  u_2> u_1 \quad  \mbox{in } \o_2, \vspace{0.2cm}\\
u_1=u_2 \quad \mbox{on } \p \o_1\cup\p\o_2.
\end{array}
\right.
$$ 
It will be clear from the proof that in suitable coordinates (after a Kelvin's transformation if necessary), $\O_1$ and $\O_2$ can be assumed to be bounded. 

\begin{rem} \label{reg}
We will not discuss here the regularity of $\o_1, \o_2$ since Theorem~\ref{thm:ineq} applies under 
very weak assumptions on $\o_1, \o_2$ as far as $\o_i\Subset \Omega_i$, $i=1,2$, see Remark \ref{smooth}. 
If $\o_i\not\Subset \Omega_i$ for some $i$, we may always choose a larger and smooth $O_i\subset \R^2$ such that $\o_i\Subset O_i$ and then 
apply Theorem~\ref{thm:ineq} with $\Omega_i$ replaced by $O_i$.
\end{rem}

\medskip

The idea is then to apply the singular Sphere Covering Inequality of Theorem~\ref{thm:ineq} in (at least one of) the $\O_i$. To this end we recall the notation $h=e^H$ and observe that,
\begin{equation} \label{Dh1}
	-\D H(x)= l\frac{4}{(1+|x|^2)^2}=\frac{4\pi\bigr( 2+\sum_j \a_j \bigr)-\rho}{4\pi}\frac{4}{(1+|x|^2)^2}\, .
\end{equation}
For fixed $\o\subseteq\R^2$ we define,
$$
	I(\o)= \dfrac{1}{4\pi}\int_{\o}\frac{4}{(1+|x|^2)^2}\,dx.
$$
Recalling the Definition \ref{def1} of $\wtilde\o$, we then set,
\begin{equation} 	\label{I}
	I_s(\o)=\left\{ \begin{array}{ll}
									I(\o) & \mbox{if $\o$ is simply-connected}, \vspace{0.2cm}\\
									I(\wtilde\o) & \mbox{if $\o$ is not simply-connected}.
								\end{array} \right.	
\end{equation}
We will prove $(i)$ and $(ii)$ of Theorem \ref{thm:mf} separately.

\

\textbf{Proof of $(i)$.} We have $N\geq 0$ and $\rho<4\pi\bigr( 2+\sum_j \a_j \bigr)$. Observe that in this case $-\D H>0$ in $\R^2$. Thus, by the definition of $\a(\o_i)=\a(\o_i,h,f)\geq 0$ in \eqref{a}, for \eqref{eq-ui} we have,
\begin{equation} \label{a_i}
	\a(\o_i)= \frac{4\pi\bigr( 2+\sum_j \a_j \bigr)-\rho}{4\pi}\,I_s(\o_i)-\sum_{j\in J_i} \a_j, \quad i=1,2,
\end{equation}
where   $J_i=\bigr\{j\in\{1,\dots,N\}\,:\, q_j\in \o_i\bigr\}$. Moreover, let us define,
$$
	\a=\a(\o_1)+\a(\o_2).
$$
We point out that obviously,
\begin{equation} \label{i-tot}
	I_s(\R^2)= 1 = \dfrac{|\S^2|}{4\pi}\,.
\end{equation}
Therefore, we have
\begin{align} 
	\a &=\frac{4\pi\bigr( 2+\sum_j \a_j \bigr)-\rho}{4\pi}\bigr(I_s(\o_1)+I_s(\o_2)\bigr)-\sum_{j\in J_1  \cup J_2} \a_j  \nonumber\\
	   &\leq  \frac{4\pi\bigr( 2+\sum_j \a_j \bigr)-\rho}{4\pi}\,I_s(\R^2)-\sum_{j\in J_1\cup J_2} \a_j \nonumber\\
		 &=\frac{4\pi\bigr( 2+\sum_j \a_j \bigr)-\rho}{4\pi}-\sum_{j\in J_1\cup J_2}  \a_j  \nonumber\\
		 &\leq \frac{4\pi\bigr( 2+\sum_j \a_j \bigr)-\rho}{4\pi}-\sum_{j}  \a_j  \nonumber\\		
		 &=\frac{8\pi-\rho}{4\pi}\,, \label{sum}
\end{align}
and then, in particular, it holds, 
\begin{equation} \label{sum2}
	\a<2.
\end{equation}

We claim that
\begin{equation} \label{cl}
	2\rho > 16\pi - 8\pi\a,
\end{equation}
and discuss two cases separately.

\medskip
{\bf CASE 1}. Suppose first that $\a(\o_i)<1$ for all $i=1,2$.
By applying Theorem~\ref{thm:ineq} in $\O_1$ and in $\O_2$ we find that,
\begin{equation} \label{ineq-oi}
		\int_{\o_i} \bigr( h(x)e^{u_1}+h(x)e^{u_2} \bigr)\,dx \geq 8\pi(1-\a(\o_i)), \quad i=1,2.
\end{equation}
Observe that the above inequalities are strict. Indeed, let us focus for example on $i=1$. In view of Theorem~\ref{thm:ineq}, the equality holds if and only if, in suitable coordinates, $\o_1=B_{\d}(0)$ for some $\d>0$ and $h(x)e^{u_i}\equiv |x|^{-2\a}e^{U_{\l_i,\a}}$ for some $\l_i>0$, $i=1,2$,
where $\a=\frac{1}{4\pi}\mu_+(\o_1)$ and $U_{\l_i,\a}$ is defined in \eqref{U}. Moreover, $\mu_+(\o_1)=-\D H-f=4\pi\a\d_{p=0}$. Therefore, we have $u_i+H(x)\equiv U_{\l_i,\a}-2\a\ln|x|$ and in particular
$$
     \D u_i +\D H(x)=\D U_{\l_i,\a}-4\pi\a\d_{p=0},\mbox{ in } B_\d(0).
$$
By using first \eqref{eq:U}, \eqref{eq-ui} and then $h(x)e^{u_i}\equiv |x|^{-2\a}e^{U_{\l_i,\a}}$ we come up with,
$$
      4\pi\sum_{j\in J_1}\a_j \d_{q_j}+\Delta H(x)=-4\pi\a\d_{p=0},\mbox{ in } B_\d(0),
$$
where $J_1$ is defined after \eqref{a_i}. By \eqref{Dh1} we have,
$$
	\frac{4\pi\bigr( 2+\sum_j \a_j \bigr)-\rho}{4\pi}\frac{4}{(1+|x|^2)^2} = 4\pi\a\d_{p=0}+4\pi\sum_{j\in J_1}\a_j \d_{q_j},\mbox{ in } B_\d(0),
$$
which is impossible since we also have $\rho<4\pi\bigr( 2+\sum_j \a_j \bigr)$. We conclude that the inequalities in \eqref{ineq-oi} are strict.

At this point we recall that \eqref{tot} holds true and that $\O_1, \O_2$ are disjoint regions. Therefore, since $\o_i\subseteq\O_i$, $i=1,2,$ by summing up the above inequalities we deduce 
that,
\begin{align*}
	2\rho & = \int_{\R^2} \bigr( h(x)e^{u_1}+h(x)e^{u_2} \bigr)\,dx \geq \sum_{i=1}^2 \int_{\o_i} \bigr( h(x)e^{u_1}+h(x)e^{u_2} \bigr)\,dx \\
				& > 16\pi - 8\pi\bigr(\a(\o_1)+\a(\o_2) \bigr),
\end{align*}
which proves the claim.

\medskip

{\bf CASE 2}. We assume without loss of generality that $\a(\o_1)\geq 1$.
Then, we have
$$
	\a(\o_2)=\a-\a(\o_1)\leq \a-1.
$$ 
Moreover, in view of \eqref{sum2} we deduce that $\a(\o_2)<1$. By applying Theorem~\ref{thm:ineq} in $\o_2$ and by the latter estimate we obtain 
\begin{align*}
	2\rho & = \int_{\R^2} \bigr( h(x)e^{u_1}+h(x)e^{u_2} \bigr)\,dx > \int_{\o_2} \bigr( h(x)e^{u_1}+h(x)e^{u_2} \bigr)\,dx \\
				& \geq 8\pi(1 - \a(\o_2))\geq 16\pi - 8\pi\a,
\end{align*}
which proves the claim.

\

By using \eqref{cl} together with \eqref{sum} we end up with
$$
	2\rho > 16\pi - 8\pi\a > 2\rho,
$$	
which is the desired contradiction. The proof of $(i)$ is concluded.

\

\textbf{Proof of $(ii)$.} Here we have $N\geq 3$ and $\rho=4\pi\bigr( 2+\sum_j \a_j \bigr)$. Therefore, in view of \eqref{tot},
we see that a necessary condition for the existence of a solution is,
\begin{equation} \label{necess}
	\sum_{j} \a_j>-2.
\end{equation}
\noindent

Therefore, from now on we assume without loss of generality that \eqref{necess} holds true. In view of  \eqref{Dh1}, we find that 
$-\D H=0$ in $\R^2$. Hence, by the definition of $\a(\o_i)=\a(\o_i,h,f)\geq 0$ in \eqref{a}, for the equation \eqref{eq-ui} we have,
\begin{equation} \label{a_i2}
	\a(\o_i)= -\sum_{j\in J_i} \a_j, \quad i=1,2,
\end{equation}
where $J_i=\bigr\{j\in\{1,\dots,N\}\,:\, q_j\in \o_i\bigr\}$. As above we set,
$$
	\a=\a(\o_1)+\a(\o_2).
$$
Observe that, in view of  \eqref{necess}, we have,
\begin{equation} \label{aaa}
	\a=-\sum_{j\in J_1\cup J_2} \a_j \leq -\sum_{j} \a_j <2.
\end{equation}

We claim that
\begin{equation} \label{cl2}
	2\rho > 16\pi + 8\pi\sum_{j} \a_j.
\end{equation}
By arguing as in CASE 2 of $(i)$ above and by using \eqref{aaa} we find that,
$$
	2\rho > 16\pi-8\pi\a  \geq 16\pi+8\pi \sum_{j} \a_j,
$$
whenever $\a(\o_i)\geq 1$ for some $i\in\{1,2\}$, where $\a(\o_i)$ is defined in \eqref{a_i2}. Therefore, we may restrict to the case $\a(\o_i)<1$ for all $i=1,2$. By applying Theorem~\ref{thm:ineq} in $\O_1$ and in $\O_2$ we find that,
\begin{equation} \label{ineq-oi2}
		\int_{\o_i} \bigr( h(x)e^{u_1}+h(x)e^{u_2} \bigr)\,dx \geq 8\pi(1-\a(\o_i)), \quad i=1,2.
\end{equation}
Next, in view of \eqref{tot}, and since $\O_1, \O_2$ are disjoint regions and $\o_i\subseteq\O_i$, $i=1,2,$ by summing up the above inequalities we deduce that,
\begin{align*}
	2\rho & = \int_{\R^2} \bigr( h(x)e^{u_1}+h(x)e^{u_2} \bigr)\,dx \geq \sum_{i=1}^2 \int_{\o_i} \bigr( h(x)e^{u_1}+h(x)e^{u_2} \bigr)\,dx \\
				& \geq 16\pi - 8\pi\bigr(\a(\o_1)+\a(\o_2) \bigr)= 16\pi-8\pi\a \\
				& \geq 16\pi+8\pi \sum_{j} \a_j,
\end{align*}
where we have used \eqref{aaa} in the last step. To prove the claim \eqref{cl2} it is enough to prove one of the above inequalities is strict. The equality in the last step holds if and only if $\a=-\sum_{j} \a_j$, that is, by the definition of $\a(\o_i)$ in \eqref{a_i2}, if and only if 
$$ 
	\{q_1,\dots,q_N\}\subset \o_1\cup \o_2.
$$ 
Since by assumption $N\geq 3$ we conclude that at least one of $\o_1, \o_2$ must contain more than one of the points $q_j$'s. Without loss of generality we assume 
$$ 
	|J_1|>1,
$$	
where $J_i$ is defined after \eqref{a_i2}. We shall prove that in this case the inequality in \eqref{ineq-oi2} is strict for $i=1$. Indeed, in view of Theorem~\ref{thm:ineq}, the equality holds if and only if, in suitable coordinates, $\o_1=B_{\d}(0)$ for some $\d>0$ and $h(x)e^{u_i}\equiv |x|^{-2\a}e^{U_{\l_i,\a}}$ for some $\l_i>0$, $i=1,2$,
where $\a=\frac{1}{4\pi}\mu_+(\o_1)$ and $U_{\l_i,\a}$ is defined in \eqref{U}. Moreover, $\mu_+(\o_1)=-\D H-f=4\pi\a\d_{p=0}$. Therefore, we have $u_i+H(x)\equiv U_{\l_i,\a}-2\a\ln|x|$ and in particular
$$
     \D u_i +\D H(x)=\D U_{\l_i,\a}-4\pi\a\d_{p=0},\mbox{ in } B_\d(0).
$$
By using first \eqref{eq:U}, \eqref{eq-ui} and then $h(x)e^{u_i}\equiv |x|^{-2\a}e^{U_{\l_i,\a}}$ we find that,
$$
        \sum_{j\in J_1}\a_j \d_{q_j}+\Delta H(x)=-4\pi\a\d_{p=0},\mbox{ in } B_\d(0).
$$
Recalling that $-\D H=0$ in $\R^2$ and that $|J_1|>1$, we readily conclude that the latter equality is impossible. 
Therefore, in particular the inequality in \eqref{ineq-oi2} is strict for $i=1$, which proves the claim \eqref{cl2}. 

\

At this point, since $\rho=4\pi\bigr( 2+\sum_j \a_j \bigr)$, then, by using \eqref{cl2}, we deduce that,
$$
	2\rho > 16\pi+8\pi \sum_{j} \a_j =2\rho,
$$
which is the desired contradiction. The proof is completed.
\end{proof}

\

\section{Symmetry for Spherical Onsager vortex equation} \label{sec:onsager}

In this section we provide the proof of the symmetry result for the spherical Onsager vortex equation \eqref{OnsagerVortexPDE} of Theorem \ref{thm:onsager}, see subsection \ref{subsec:onsager}. The argument is based on the singular Sphere Covering Inequality, Theorem \ref{thm:ineq}, jointly with some ideas introduced in \cite{GM1}.

\medskip

\noindent
\begin{proof}[Proof of Theorem \ref{thm:onsager}]
Without loss of generality we may assume that $\vec{n}$ coincides with the north pole, i.e. $\vec{n} =\mathcal N=(0,0,1)$. 
Let $\Pi: \S^2\setminus\{\mathcal N\} \rightarrow \R^2$ be the stereographic projection with respect to the north pole defined in \eqref{Pi} and 
$v$ be a solution of \eqref{OnsagerVortexPDE} with $8\pi<\beta\leq 16\pi$ and $\gamma$ as in \eqref{improvedEstimate}. By setting,
$$
u(x)=v(\Pi^{-1}(x)) \quad \mbox{for }x\in \R^2,
$$
then $u$ satisfies, 
\begin{equation}
\Delta u+\frac{J^2(x)\exp\bigr(\beta u-\gamma \psi (x)\bigr)}{\int_{\R^2}J^2(x)\exp\bigr(\beta u-\gamma \psi (x)\bigr) \,dx }-\frac{J^2(x)}{4\pi}=0 \quad \mbox{in } \R^2, 
\end{equation}
where 
$$
J(x)=\frac{2}{1+|x|^2} \qquad \hbox{and} \qquad \psi(x)=\frac{|x|^2-1}{|x|^2+1}\,.
$$
As in \cite{GM1, Lin1} we define,
$$
w(x)=\beta\left(u(x)-\frac{1}{4 \pi} \ln\left(1+|x|^2\right) \right)-c,
$$
with
$$
	c=\gamma +\ln \left( \frac{2}{\beta}\int_{\R^2} J^2(x)e^{\beta u-\gamma \psi}\,dx\right).
$$
Then, we have,
\begin{equation}\label{OnsagerInR2}
\Delta w+ h(x)e^{w}=0 \quad \mbox{in } \R^2,
\end{equation}
and, 
$$
\int_{\R^2}h(x)e^{w}\,dx=\beta,
$$
where, 
\begin{equation}\label{onsager}
h(x)=h(|x|)=8(1+|x|^2)^{\left(-2+\frac{\beta}{4\pi}\right)}e^{\gamma J(x)},
\end{equation}
which, letting $h=e^H$, satisfies, 
\begin{equation} \label{Dh2}
\Delta H(x) =\frac{4\left(-2+\frac{\beta}{4 \pi}\right)}{(1+|x|^2)^2}+\frac{8 \gamma(|x|^2-1)}{(1+|x|^2)^3}.
\end{equation}
Observe that $-\Delta H \leq 0$ in $\R^2$ if and only if $\gamma \leq \frac{\beta}{8\pi}-1$. For the latter range of $\gamma$ we already know from Theorem E that every solution to \eqref{OnsagerVortexPDE} is axially symmetric with respect to $\vec{n}$. Therefore, we suppose from now on that $\gamma > \frac{\beta}{8\pi}-1$. Then, $-\Delta H > 0$ in $B_{r}(0)$, where
\begin{equation} \label{r}
r^2=\frac{\gamma+1-\frac{\beta}{8\pi}}{\gamma-1+\frac{\beta}{8\pi}}\,.
\end{equation}
Our goal is to show that the solution $v$ to \eqref{OnsagerVortexPDE} is evenly symmetric with respect to a plane passing through the origin and containing 
the vector $\vec{n}$., i.e. that $w$ is evenly symmetric with respect to a line passing through the origin and a point $p\in\R^2$. 
First of all, observe that $\lim_{|x|\to+\infty} w(x)=-\infty$ and hence $w$ has a maximum point which we denote by $p\in\R^2$. 
Without loss of generality we may assume that $p$ lies on $x_1$-axis, and then define $w^*(x_1,x_2)=w(x_1,-x_2)$ and
$$
	\wtilde w=w-w^*.
$$
With these notations we are left to prove that $\wtilde w\equiv 0$ in $\R^2$ and we assume by contradiction this is not the case. Observe that $\wtilde w$ satisfies, 
\begin{equation} \label{w-tilde}
	\D \wtilde w + c(x)\wtilde w= 0 \quad \mbox{in } \R^2, \qquad c(x)=h(x)\frac{e^{w}-e^{w^*}}{w-w^*}\,.
\end{equation}
On the other hand $\wtilde w(x_1,0)=0$ for all $x_1\in\R$. Moreover, $\wtilde w$ has a critical point at $p$ that lies on the $x_1$-axis and 
thus $\wtilde w_{x_2}(p)=0$.\\ We claim that $\wtilde w$ changes sign in $\R^2_+$. 
Indeed, if this were not the case, then we could assume that $\wtilde w<0$ in any $B_R^+(p)=\{x\in B_R(p)\,:\, x_2>0\}$. 
However, the latter fact jointly with $\wtilde w_{x_2}(p)=0$ contradicts the thesis of Hopf's Lemma when applied to the equation
\eqref{w-tilde} at the point $p$ in the domain $B_R^+(p)$.\\ 
As a consequence, we conclude that there exist at least two disjoint simply-connected regions $\O_i \subset\R^2_+$, $i=1,2$ 
(not necessarily bounded) such that there exist a pair of open subsets $\o_i\subseteq\O_i$, $i=1,2,$ such that,
$$	
\left\{ \begin{array}{l}
w> w^* \quad  \mbox{in } \o_1, \qquad  w^*> w \quad  \mbox{in } \o_2, \vspace{0.2cm}\\
w=w^* \quad \mbox{on } \p \o_1\cup\p\o_2.
\end{array}
\right.
$$
By arguing as in the proof of Theorem \ref{thm:mf}, we may assume without loss of generality that $\O_1$ and $\O_2$ are bounded. As for 
the regularity of $\o_1, \o_2$ we refer to Remark~\ref{reg}. By applying Theorem \ref{thm:ineq} in both $\o_i$, $i=1,2$, and summing up we get,
\begin{align*}
\beta &= \int_{\R^2}h(x)e^{w} \,dx = \int_{\R^2_+}\left(h(x)e^{w}+h(x)e^{w^*}\right) \,dx \\
&\geq \sum \limits_{i=1}^{2} \int_{\omega_i} \left(h(x)e^{w}+h(x)e^{w^*}\right) \,dx >8\pi \left( 2- \alpha(\omega_1)-\alpha(\omega_2)\right).
\end{align*}
One can show that the last inequality is strict by using the same argument as in the proof of Theorem~\ref{thm:mf}. Hence,
\begin{equation} \label{condition}
\beta + 8 \pi (\alpha(\omega_1)+\alpha(\omega_2)) > 16 \pi.
\end{equation}
On the other hand, recalling the definition of $\a$ in \eqref{a} (see also Example \ref{ex1}), the expression in \eqref{Dh2} and the discussion right above \eqref{r}, we find that,  
\begin{align*}
 8 \pi (\alpha(\omega_1)+\alpha(\omega_2)) &\leq -2\int_{B_r^+(0)} \left(\frac{4\left(-2+\frac{\beta}{4 \pi}\right)}{(1+|x|^2)^2}+\frac{8 \gamma(|x|^2-1)}{(1+|x|^2)^3}\right) \,dx \\
 & = -\int_{B_r(0)} \left(\frac{4\left(-2+\frac{\beta}{4 \pi}\right)}{(1+|x|^2)^2}+\frac{8 \gamma(|x|^2-1)}{(1+|x|^2)^3}\right) \,dx \\
 & \\
 &= -\left(4\left(-2+\frac{\beta}{4 \pi}\right)+8\gamma\right) \int_{B_r(0)} \frac{dx}{(1+|x|^2)^2}+ 16 \gamma  \int_{B_r(0)} \frac{dx}{(1+|x|^2)^3} \\
 & \\
 &= -8\pi\left(-1+\frac{\beta}{8 \pi}+\gamma\right) \left(1-\frac{1}{1+r^2}\right)+ 8 \pi \gamma \left(1-\frac{1}{(1+r^2)^2}\right)\\
 & \\
 &= 8\pi \left(1-\frac{1}{1+r^2}\right)\left(-\left(-1+\frac{\beta}{8\pi}+\gamma\right) +\gamma\left(1+\frac{1}{1+r^2}\right) \right).
\end{align*}
In view of \eqref{r} we also have,
\begin{align*}
 8 \pi (\alpha(\omega_1)+\alpha(\omega_2)) &=8\pi \left(1-\frac{1}{1+r^2}\right)\left(-\left(-1+\frac{\beta}{8\pi}+\gamma\right) +\gamma\left(1+\frac{1}{1+r^2}\right) \right)\\
 & \\
 &= 8\pi \left( \frac{\gamma+1-\frac{\beta}{8\pi}}{2\gamma}\right) \left(-\left(-1+\frac{\beta}{8\pi}+\gamma\right) +\frac{3\gamma-1+\frac{\beta}{8\pi}}{2} \right)\\
 &= 8\pi \frac{\left(\gamma+1-\frac{\beta}{8\pi}\right)^2}{4\gamma}\,.
\end{align*}
Inserting the latter estimate into \eqref{condition} we end up with,
$$
\frac{\beta}{8\pi} + \frac{(\gamma+1-\frac{\beta}{8\pi})^2}{4\gamma} > 2. 
$$
The above inequality can be rewritten as 
\begin{equation*}
\gamma^2+2\gamma\left(\frac{\beta}{8\pi}-3\right)+\left(\frac{\beta}{8\pi}-1\right)^2> 0,
\end{equation*}
which contradicts the assumption \eqref{improvedEstimate} on $\gamma$. It follows that $w$ is evenly symmetric about the line passing 
through the origin and the point $p\in\R^2$, as claimed.
\end{proof}

\

\end{document}